\documentclass[11pt,a4paper]{article}
\usepackage[T1]{fontenc}
\usepackage[utf8]{inputenc}
\usepackage[english]{babel}
\usepackage{lmodern}
\usepackage[margin=29mm]{geometry}
\usepackage{amsmath,amssymb,amsthm,mathtools,mathrsfs,bm}
\usepackage[numbers,sort&compress]{natbib}
\usepackage{booktabs,longtable,array}
\usepackage{microtype}
\usepackage[colorlinks=true,linkcolor=blue,citecolor=blue,urlcolor=blue]{hyperref}
\hypersetup{pdftitle={Weak limits of multiple Wiener integrals:
representation, extraction, and realization},pdfauthor={Obayda Julien
Assaad}}
\allowdisplaybreaks[2]
\theoremstyle{plain}
\newtheorem{theorem}{Theorem}[section]
\newtheorem{lemma}[theorem]{Lemma}
\newtheorem{proposition}[theorem]{Proposition}
\newtheorem{corollary}[theorem]{Corollary}
\theoremstyle{definition}
\newtheorem{definition}[theorem]{Definition}

\newtheorem{example}[theorem]{Example}
\newtheorem{algorithm}[theorem]{Procedure}
\theoremstyle{remark}
\newtheorem{remark}[theorem]{Remark}
\newcommand{\E}{\mathbb E}
\newcommand{\R}{\mathbb R}

\newcommand{\1}{\mathbf 1}
\newcommand{\Id}{\operatorname{Id}}

\newcommand{\Ker}{\operatorname{Ker}}
\newcommand{\Span}{\operatorname{span}}

\newcommand{\Var}{\operatorname{Var}}
\newcommand{\Tr}{\operatorname{Tr}}
\newcommand{\Sym}{\operatorname{Sym}}
\newcommand{\HS}{\mathrm{HS}}
\newcommand{\op}{\mathrm{op}}
\newcommand{\Law}{\mathcal L}

\newcommand{\cD}{\mathcal D}
\newcommand{\cG}{\mathcal G}
\newcommand{\cH}{\mathcal H}

\newcommand{\cP}{\mathcal P}

\newcommand{\wtensor}{\widetilde\otimes}
\newcommand{\norm}[1]{\left\lVert #1\right\rVert}

\newcommand{\ip}[2]{\left\langle #1,#2\right\rangle}
\newcommand{\weak}{\Longrightarrow}
\newcommand{\wt}{\operatorname{wt}}

\newcommand{\eps}{\varepsilon}

\numberwithin{equation}{section}
\title{Weak limits of multiple Wiener integrals:\\representation, extraction, and realization}
\author{Obayda Julien Assaad\\\small Independent researcher\\\small\texttt{Obayda.assaad@gmail.com}}
\date{}

\begin{document}
\maketitle
\begin{abstract}
We study weak limits of multiple Wiener integrals together with the
limits of their kernel contractions. After passage to a subsequence,
bounded kernel sequences admit a Gaussian Hermite representation whose
coordinates carry weights determined by the original chaos orders.
In the limits of kernel contractions, only pairings of whole Gaussian
coordinates of equal weight contribute. We construct kernel projections with explicit bounds for
Gaussian approximation and asymptotic independence. For homogeneous
inputs, spectral extraction converges to the decomposable and primitive
projections in the limiting kernel spaces. Conversely, prescribed
weighted Hermite expansions are realized by smooth compactly supported
kernels away from the diagonals, with exact inner products for finite
expansions and asymptotically prescribed contractions.
We then classify fixed-chaos limits with an absolute exponential moment.
For a standard Gaussian vector $X$ and a symmetric matrix $\mathsf A$,
a quadratic target $X^{\mathsf T}\mathsf A X-\Tr\mathsf A+b^{\mathsf T}X$
is such a limit if and only if $\mathsf A b=0$. We also give a finite
cumulant criterion and a joint version.
Finally, we identify limits under stationary Ornstein--Uhlenbeck noise:
a coordinate of weight $s$ has relaxation rate $s$. Two fourth-chaos
examples have identical one-time laws and covariance functions but
different asymptotic variances for time averages of their squares.
For kernels with nonnegative coordinates, contractions of the
weight-one component along simple regular graphs are nonnegative;
this excludes an explicit cubic representation that is attained by
signed kernels.
\end{abstract}



\section{Introduction}
\label{intro:section}

Let $F_n=I_q^{W_n}(f_n)$ be a sequence of multiple Wiener integrals of
fixed order $q$, with bounded second moments. We study its weak limits
while retaining the limits of contractions of the kernels $f_n$.
We also consider vectors whose components belong to a fixed finite
sum of Wiener chaoses. Our results give a simultaneous Gaussian
representation of these limits, quantitative kernel decompositions,
and constructions of realizing sequences.

We write
\begin{equation}
 x_n=\sqrt{q!}\,f_n,\qquad
 F_n=I_q^{W_n}(x_n/\sqrt{q!}),\qquad
 \E F_n^2=\|x_n\|^2.
 \label{intro:normalization}
\end{equation}
The multiple integral $I_q$ has its usual Wiener--It\^o normalization.

\subsection{Main results}

Every weak limit from the homogeneous chaos of order $q$ has a
representation of the form
\[
 F=\sum_{\sum_{s,j}s\alpha_{s,j}=q}
 c_\alpha\prod_{s,j}H_{\alpha_{s,j}}(G_{s,j}),
 \qquad \sum_\alpha c_\alpha^2\alpha!<\infty.
\]
Here the $G_{s,j}$ are independent standard Gaussians, each multi-index
has finite support, $\alpha!=\prod_{s,j}\alpha_{s,j}!$, and the series
converges in $L^2$. The $H_k$ are the probabilists' Hermite polynomials. The coordinate
$G_{s,j}$ has weight $s$. Thus the
weighted degree of every term is $q$, although its ordinary chaos order
is $\sum_{s,j}\alpha_{s,j}$. Theorem~\ref{rec:closure} states the resulting weak-closure
characterization. Theorems~\ref{str:representation}
and~\ref{rec:realization} give a simultaneous representation and
realization for the prescribed inputs.

The representation is constructed from Hilbert spaces $\cH_q$ of
kernel sequences. The closed span $\cD_q$ of products of lower orders
has orthogonal complement $\cG_q$. Elements of $\cG_q$, called
primitive elements, have vanishing proper contractions and give
Gaussian coordinates of weight $q$. Theorem~\ref{str:contractions}
identifies the limits of contraction diagrams: only pairings of entire
primitive factors of equal weight contribute. For tensor-valued
outputs, the statement concerns their limiting norms and scalar
products.

The extraction results act on the original kernels. For $d$ components
of orders at most $Q$, assume
$\sum_{a,q}\|x_{a,q}\|^2\le R^2$.
The procedure of Theorem~\ref{ext:finite} gives, at every integer
level $J\ge1$, subspaces satisfying
\begin{equation}
 \dim M_q^{[J]}\le d(J-1),\qquad
 D_R+D_I\le\frac{\mu_Qa_QR^4}{J}.
 \label{intro:finite-rate}
\end{equation}
Here $D_R$ is the squared $L^2$ error of the residual
Malliavin--Stein matrix relative to its covariance, and $D_I$ is the
sum of squared mixed Malliavin--Stein terms. Their definitions and
the constants $a_Q,\mu_Q$ are given in Section~\ref{ext:section}.
These bounds imply rates $J^{-1/2}$ for smooth test functions and
$J^{-1/4}$ in $W_1$, without assuming an invertible residual covariance.

For a homogeneous order-$q$ input, the distinct spectral procedure of
Theorem~\ref{ext:reduced} satisfies
\begin{equation}
 \dim M_q^{[J]}\le d(J-1),\qquad
 D_R+D_I\le\frac{d a_qR^4}{J}.
 \label{intro:spectral-rate}
\end{equation}
Its active and residual projections converge in the kernel sequence
space to $P_{\cD_q}\xi$ and $P_{\cG_q}\xi$, respectively
(Theorem~\ref{ext:homogeneous}). These dimension bounds concern
kernel subspaces, not the number of Gaussian coordinates in the
representation. The approximation rates above do not assert a
uniform rate of convergence to the limiting projections.

The converse construction in Section~\ref{rec:section} uses smooth
compactly supported kernels on $L^2(\R)$, supported away from the
spatial diagonals. It preserves inner products exactly for every
finite Hermite expansion and preserves fixed contractions
asymptotically. For infinite expansions, the exact inner-product
identity applies at each finite orthogonal truncation.

\subsection{Applications}

Our first application classifies limits with an absolute exponential
moment. If a fixed-chaos limit $F$ satisfies
$\E e^{c|F|}<\infty$ for some $c>0$, its Gaussian representation
has ordinary degree at most two. Theorem~\ref{light:classification}
then gives the complete lists
\[
 \sigma G_0+\sum_j\lambda_j(G_j^2-1)
 \quad\hbox{for every even order }q\ge2,
\]
with $\sum_j\lambda_j^2<\infty$, and
\[
 \sigma G_0+\sum_jb_jG_j\widetilde G_j
 \quad\hbox{for every odd order }q\ge3,
\]
with $\sum_jb_j^2<\infty$. Here $\sigma\ge0$, and all displayed
Gaussian variables are independent and standard. The exponential-moment
assumption is imposed only on the limit.

For a finite-dimensional quadratic Gaussian target, this classification
gives the exact condition
\begin{equation}
 \begin{split}
 &X^{\mathsf T}\mathsf A X-\Tr\mathsf A+b^{\mathsf T}X
   \text{ is a weak limit in some fixed homogeneous chaos}\\
 &\hspace{35mm}\Longleftrightarrow\quad \mathsf A b=0.
 \end{split}
 \label{intro:quadratic-test}
\end{equation}
Here $X$ is standard Gaussian and $\mathsf A$ is symmetric.
Theorem~\ref{light:quadratic} also gives a Hankel matrix criterion
using cumulants through order $4t+4$, where $t$ is the number of
distinct nonzero eigenvalues of $\mathsf A^2$.
Krein~\cite[Remark~4.10(2)]{Krein2019} discusses the possible
occurrence of noncentral quadratic Gaussian limits in a fixed chaos;
the criterion above excludes them when $\mathsf A b\ne0$.
Theorem~\ref{light:vector} gives the corresponding joint criterion.

Our second application identifies limits under stationary
Ornstein--Uhlenbeck noise. A primitive coordinate of weight $s$
converges to a stationary Gaussian process with covariance
$e^{-s|t-u|}$. Theorem~\ref{app:process} gives convergence of the
represented observables in $C([0,T],\R^d)$ for every finite $T$.
The limiting Gaussian generator is
\[
 L_E=\sum_s sL_s.
\]
The weights determine process information that is not determined by
the marginal law. In Proposition~\ref{app:fourth-example}, two
fourth-chaos sequences converge to products of independent Gaussian
processes with rates $(1,3)$ and $(2,2)$. Both limits have the same
one-time law and covariance $e^{-4|t-u|}$, but the asymptotic variances
of the time averages of their centered squares are $11/3$ and $3$.
These values follow from the spectral variance formula in the
Kipnis--Varadhan theorem~\cite{KipnisVaradhan}.

Section~\ref{pos:section} gives a further application to kernels
with nonnegative coordinates. Positivity of regular graph contractions
restricts their components over the weight-one coordinate space $E_1$. The obstruction concerns a
prescribed tensor representation, rather than every representation of
the same scalar law.

\subsection{Related work}

Herry, Malicet and Poly~\cite{HerryMalicetPoly} prove closure in law
of Gaussian polynomials of bounded degree, jointly for countably many
coordinates. Their Lemmas~3.13 and~3.18 track Malliavin operations
and asymptotic chaos orders. Our formulation makes the limiting
kernel contractions explicit and provides extraction converging to
identified orthogonal projections, together with a simultaneous
isometric realization.

The closure of the second Wiener chaos is classical
\cite{Sevastyanov,Arcones,NourdinPolySecond,NourdinPolyErratum}.
The exponential-moment classification identifies its role in the
closures of all higher even homogeneous chaoses and describes the
additional restriction at odd orders.

Quantitative decompositions of Gaussian polynomials were obtained by
De and Servedio~\cite{DeServedio} and by De, Mossel and
Neeman~\cite{DeMosselNeeman}. Their regularization results control
factor representations. Our dimension bounds instead concern the
kernel subspaces selected by the stated procedures.

The Gaussian convergence and independence arguments use the fourth
moment theorem, its multivariate extension, and contraction criteria
\cite{NualartPeccati,NualartOrtiz,PeccatiTudor,UstunelZakai,NourdinRosinski}.
The quantitative comparison uses Gaussian interpolation and
Malliavin--Stein estimates
\cite{NourdinPeccatiStein,NourdinPeccatiReveillac,
NourdinNualartPeccatiStable,TudorQuantitativeIndependence}.
Joint convergence involving iterated Malliavin covariances is also
studied in~\cite{TudorJointConvergence}.
The distinct-index Hermite limits used in the realization have
classical antecedents~\cite{DynkinMandelbaum,AvramTaqqu}.
We use the product formula, hypercontractivity and diagram estimates
in their standard forms~\cite{Janson,Nualart,PeccatiTaqqu,MajorLectures}.

Section~\ref{framework:section} fixes the notation, summarized in
Table~\ref{notation:table}, and states the representation theorems.
Sections~\ref{str:proofs-section}, \ref{ext:section} and
\ref{rec:section} establish the representation, extraction and
realization results. Sections~\ref{light:section}, \ref{app:section}
and~\ref{pos:section} develop the applications. The final section
collects further questions.

\section{Preliminaries and representation theorems}
\label{pre:section}\label{framework:section}

Fix a maximal degree $Q<\infty$, enlarged in advance to include all
expressions under consideration. All Hilbert spaces are real and
separable, and $H^{\otimes0}=\R$. We use the multiple integrals of \cite{Wiener,Ito}, with
the conventions of \cite[Chapter~1]{Nualart}. Throughout the article a
kernel of order $q$ is rescaled by
\begin{equation}
 x=\sqrt{q!}\,f,\qquad
 F=I_q\!\left(\frac{x}{\sqrt{q!}}\right),\qquad
 \E[F^2]=\|x\|^2.
 \label{pre:isometry}
\end{equation}
For $q\ge1$, the squared norm of $x$ is therefore the variance of
its multiple integral. The contraction $x\otimes_r y$ is the usual
Hilbert tensor contraction;
$x\wtensor_r y$ denotes their orthogonal symmetrization. In particular,
\begin{equation}
 \|x\wtensor_r y\|\le\|x\otimes_r y\|\le\|x\|\|y\|,
 \qquad
 x\diamond y=\sqrt{\frac{(p+q)!}{p!q!}}\Sym_{p+q}(x\otimes y)
 \label{pre:contraction-bound}
\end{equation}
for $x\in H^{\odot p}$ and $y\in H^{\odot q}$.
The operation $\diamond$ is associative and represents the top-degree
term of a product of multiple integrals.

\begin{lemma}[Product formula and hypercontractivity]
\label{pre:product}\label{pre:hypercontractivity}
For $k=p+q-2r$ and
$\beta_{p,q,r}=r!\binom pr\binom qr\sqrt{k!/(p!q!)}$,
\begin{equation}
 I_p\!\left(\frac{x}{\sqrt{p!}}\right)
 I_q\!\left(\frac{y}{\sqrt{q!}}\right)
 =\sum_{r=0}^{p\wedge q}\beta_{p,q,r}
 I_k\!\left(\frac{x\wtensor_r y}{\sqrt{k!}}\right).
 \label{pre:product-formula}
\end{equation}
If $Z$ belongs to the sum of the ordinary Gaussian chaoses of orders
at most $Q$, then
\begin{equation}
 \|Z\|_{L^p}\le(p-1)^{Q/2}\|Z\|_2,\qquad p\ge2.
 \label{pre:hypercontractive}
\end{equation}
\end{lemma}
These are the classical product formula
\cite[Theorem~4.7]{NourdinLectures} and Gaussian hypercontractivity
\cite[Chapter~V]{Janson}.
\begin{lemma}[Tightness at bounded chaos order]
\label{pre:tightness}
Let $Z_n$ belong to sums of Wiener chaoses of orders at most a fixed
$Q$. If their laws are tight, then $\sup_n\|Z_n\|_2<\infty$.
If $Z_n\weak Z$, all moments converge. The same statements hold
coordinatewise, and for mixed moments, for finite-dimensional vectors.
\end{lemma}
\begin{proof}
Hypercontractivity gives $\E Z_n^4\le3^{2Q}(\E Z_n^2)^2$.
For $\|Z_n\|_2>0$, the Paley--Zygmund inequality yields
\[
 \mathbb P\left(|Z_n|>\frac{\|Z_n\|_2}{\sqrt2}\right)
 \ge\frac{(\E Z_n^2)^2}{4\E Z_n^4}
 \ge\frac1{4\,3^{2Q}}.
\]
If a subsequence of $\|Z_n\|_2$ tended to infinity, this bound
would contradict tightness. Hypercontractivity then bounds every fixed $L^p$ norm uniformly.
For a vector $Z_n=(Z_{n,1},\ldots,Z_{n,d})$ and nonnegative integer
powers $k_a$ with $k=\sum_a k_a>0$, H\"older's inequality gives
\[
 \E\left|\prod_a Z_{n,a}^{k_a}\right|^2
 \le\prod_{a:k_a>0}(\E|Z_{n,a}|^{2k})^{k_a/k}.
\]
The right side is uniformly bounded. These monomials are therefore
uniformly integrable, so weak convergence implies convergence of all
mixed moments, including the scalar case.
\end{proof}

We also use the Malliavin derivative $D$, the Ornstein--Uhlenbeck
generator $L$, and integration by parts in the form
\begin{equation}
 LI_q(x/\sqrt{q!})=-qI_q(x/\sqrt{q!}),\qquad
 \E[YZ]=\E\langle DZ,-DL^{-1}Y\rangle
 \label{pre:ibp}
\end{equation}
for centered finite-chaos $Y$ and $Z\in\mathbb D^{1,2}$
\cite[Sections~1.2 and~1.4]{Nualart}.

\subsection{Weighted Gaussian expansions}

\begin{definition}[Weighted Gaussian Hermite spaces]
\label{str:assigned-degree}
Let $E=\bigoplus_{s=1}^Q E_s$ be a direct sum of real separable
Hilbert spaces, and let $W$ be isonormal over $E$. A coordinate
$W(h)$ with $h\in E_s$ has \emph{weight} $s$. For $q\ge1$, define
$\cP_q(E)$ as the closed span in $L^2$ of the Wick products
\[
 :W(h_1)\cdots W(h_m):,
 \qquad h_i\in E_{s_i},\qquad \sum_i s_i=q,
\]
and put $\cP_0(E)=\R$. Such a product has \emph{weighted degree}
$q$ and \emph{ordinary chaos order} $m$. Thus $\cP_q(E)$ is contained
in the sum of the ordinary chaoses of orders $1,\ldots,q$; infinitely
many Gaussian coordinates may be required.
\end{definition}

\begin{theorem}[Weak limits at a prescribed chaos order]
\label{rec:closure}
For each $q\ge1$, the weak limits of multiple Wiener integrals of
order $q$ are exactly the laws of elements of $\cP_q(E)$, where
$E=\bigoplus_{s=1}^q E_s$ ranges over separable graded Gaussian
spaces. Every such law can be realized over $L^2(\R)$ by smooth
compactly supported kernels vanishing whenever $|t_i-t_j|<1$ for
some $i\ne j$.
A realization is also possible over any sequence of Hilbert spaces
whose dimensions tend to infinity; the smoothness assertion pertains
to the spatial model.
\end{theorem}
The representation is proved in Section~\ref{str:proofs-section};
the converse construction and the proof of this theorem are given in
Section~\ref{rec:section}.

\subsection{Kernel sequence spaces and primitive elements}
\label{str:spaces-section}

\begin{proposition}[Hilbert spaces of kernel sequences]
\label{str:sequence-spaces}
Given a countable family of bounded sequences
$x_n\in H_n^{\odot q}$, $1\le q\le Q$, one can select a subsequence
and construct separable Hilbert spaces $\cH_q$ containing their classes
such that
\begin{equation}
 \langle\xi,\eta\rangle_{\cH_q}
 =\lim_n\langle x_n,y_n\rangle.
 \label{str:limit-inner-product}
\end{equation}
We put $\cH_0=\R$. Every element has a bounded representative. Products
$\diamond$, symmetrized contractions
$(x,y)\mapsto x\wtensor_r y$, and contractions against a tensor
of order $r$ define continuous operations on these spaces whenever
their output order is at most $Q$. In particular,
\begin{equation}
 \|\xi\diamond\eta\|_{\cH_{p+q}}
 \le\sqrt{\binom{p+q}{p}}\|\xi\|\|\eta\|.
 \label{str:product-bound}
\end{equation}
For every representative and $1\le r<q$, the operator norms of
\[
 T_{x_n,r}:H_n^{\odot r}\longrightarrow H_n^{\odot(q-r)},
 \qquad u\longmapsto x_n\otimes_r u,
\]
have limits. The construction
can include countably many additional bounded sequences.
\end{proposition}

\begin{definition}[Representative]
\label{str:representative}
In the preceding construction, let $\xi^{(j)}=[x^{(j)}]$ belong to
the generated dense family and converge to $\xi\in\cH_q$.
A bounded sequence $x_n\in H_n^{\odot q}$ represents $\xi$ if
\[
 \lim_{j\to\infty}\limsup_n\|x_n-x_n^{(j)}\|=0.
\]
This does not depend on the dense approximation. Two representatives
of the same element differ by a sequence whose norm tends to zero.
We write $\xi=[x]$.
\end{definition}

Define the decomposable and primitive subspaces by
\begin{equation}
 \cD_q=\overline{\Span}\{\xi\diamond\eta:
 \xi\in\cH_r,\ \eta\in\cH_{q-r},\ 1\le r<q\},\qquad
 \cG_q=\cD_q^\perp,
 \label{str:DG-definition}
\end{equation}
with $\cD_1=\{0\}$ and $\cG_1=\cH_1$.

\begin{theorem}[Primitive elements and vanishing contractions]
\label{str:contraction-criterion}
For $\xi=[x]\in\cH_q$, $q\ge2$, the following are equivalent:
\begin{equation}
 \xi\in\cG_q;\qquad
 \|T_{x_n,r}\|_{\mathrm{op}}\longrightarrow0\ (1\le r<q);\qquad
 \|x_n\otimes_r x_n\|\longrightarrow0\ (1\le r<q).
 \label{str:equivalent-conditions}
\end{equation}
They are also equivalent to $[x_n\otimes_r u_n]=0$ for every
$[u]\in\cH_r$ and $1\le r<q$.
If $[x]\in\cG_p$ and $[y]\in\cG_q$, then
\begin{equation}
 \|x_n\otimes_r y_n\|\longrightarrow0,
 \quad 1\le r\le p\wedge q,\quad \text{unless }p=q=r.
 \label{str:cross-contraction}
\end{equation}
\end{theorem}

\subsection{Isometric representation and contraction limits}
\label{str:representation-section}

\begin{theorem}[Isometric representation and joint convergence]
\label{str:representation}\label{str:law-transfer}
Set $E_s=\cG_s$ and $\mathbb U_0=\Id_\R$. There is an isometric isomorphism
$\mathbb U_q:\cP_q(E)\to\cH_q$ determined by
\begin{equation}
 \mathbb U_q(:W(h_1)\cdots W(h_m):)
 =[h_{1,n}\diamond\cdots\diamond h_{m,n}],
 \quad h_i=[h_{i,n}]\in\cG_{s_i},\quad\sum_i s_i=q.
 \label{str:U-monomial}
\end{equation}
It maps the ordinary first chaos of weighted degree $q$ onto
$\cG_q$, and its orthogonal complement onto $\cD_q$.
For any finite family $\xi_{a,q}=[x_{a,q}]$, along the selected
subsequence,
\begin{equation}
 \left(\sum_{q=1}^Q I_q^{W_n}
       (x_{n,a,q}/\sqrt{q!})\right)_{a\le d}
 \Longrightarrow
 \left(\sum_{q=1}^Q\mathbb U_q^{-1}\xi_{a,q}\right)_{a\le d}.
 \label{str:law-convergence}
\end{equation}
All joint moments converge, simultaneously for every finite subfamily.
\end{theorem}

\begin{definition}[Contraction diagram]
\label{str:diagram}
A contraction diagram $\mathfrak d$ consists of finitely many input
vertices, labeled tensor slots at each vertex, and disjoint pairs of
slots belonging to distinct vertices. Unpaired slots are output slots.
The associated multilinear map $\mathcal C_{\mathfrak d}$ contracts
the paired indices in an orthonormal basis. A closed diagram has no
output slots and gives a scalar. Self-contractions at one input vertex
are excluded.
\end{definition}

\begin{theorem}[Contraction limits by pairings of primitive factors]
\label{str:contractions}
On algebraic inputs
$p_{v,n}=h_{v,1,n}\diamond\cdots\diamond h_{v,m_v,n}$,
where $[h_{v,i}]\in\cG_{s_{v,i}}$, expand each symmetrization in
$\mathcal C_{\mathfrak d}((p_{v,n})_v)$. Retain precisely the terms
in which a contracted primitive factor is paired completely with one
factor of the same weight; other factors remain wholly uncontracted.
In the retained terms replace each full scalar contraction by
$\langle h_{v,i},h_{w,j}\rangle_{E_s}$, keeping the permutation
coefficients and the factors $\sqrt{q_v!/\prod_i s_{v,i}!}$ from
the $\diamond$ products. Call the resulting tensor $B_{\mathfrak d,n}$.
Then
\begin{equation}
 \|\mathcal C_{\mathfrak d}((p_{v,n})_v)-B_{\mathfrak d,n}\|
 \longrightarrow0.
 \label{str:open-contraction-limit}
\end{equation}
Both tensors lie in the same finite-level output space. For a closed
diagram, this gives a scalar limit. The rule extends continuously to
all sequence-space inputs through $\mathbb U_q$; open outputs are
interpreted by their sequence-space norms and scalar products.
\end{theorem}

\begin{remark}[Choice of simultaneous representation]
The spaces $\cH_q$, $\cD_q$ and $\cG_q$ refer to the selected
subsequence and the family represented jointly. Their orthogonal
projections are uniquely determined within this construction and are
preserved by the enlargements of Lemma~\ref{str:enlargement}.
This does not assert that the spaces or projections can be reconstructed
from a single marginal law.
\end{remark}

\begin{table}[ht]
\centering
\small
\caption{Principal notation.}
\label{notation:table}
\begin{tabular}{@{}p{0.29\textwidth}p{0.67\textwidth}@{}}
\toprule
Notation & Meaning\\
\midrule
$H_n,W_n$ & Original Hilbert space and isonormal process.\\
$x_n,\xi=[x]$ & Rescaled kernel and its sequence-space class.\\
$\cH_q,\cD_q,\cG_q$ & Kernel sequence space, decomposable subspace and primitive subspace.\\
$E_s,W$ & Gaussian coordinate space of weight $s$ and limiting isonormal process.\\
$\cP_q(E),\mathbb U_q$ & Weighted Hermite space and representation isometry.\\
$T_{x,r}$ & Contraction operator $u\mapsto x\otimes_r u$.\\
$M_{n,q}^{[J]}$ & Extracted kernel subspace at level $J$.\\
$F^{\mathrm{act},J},F^{\mathrm{res},J}$ & Active and residual random vectors.\\
$\mathsf B_q,\mathsf K$ & Positive extraction operator and its Gram matrix.\\
$\Gamma_{\rm MS},D_R,D_I$ & Malliavin--Stein bracket and squared residual and mixed defects.\\
$a_Q,\mu_Q$ & Coefficient and output-multiplicity bounds.\\
$L_E,\Gamma_E$ & Weighted Ornstein--Uhlenbeck generator and carr\'e du champ.\\
\bottomrule
\end{tabular}
\end{table}

\section{Proofs of the structural results}
\label{str:proofs-section}

\subsection{Construction and stability of the sequence spaces}

\begin{proof}[Proof of Proposition~\ref{str:sequence-spaces}]
Let $\mathcal B_0$ contain the prescribed sequences and the constant
sequence $1$ in order zero. Given $\mathcal B_j$, adjoin its rational
linear combinations, all products and contractions with output order
at most $Q$, and the following sequences. For each order-$q$ member
$x$ and $1\le r<q$, choose $u_n\in H_n^{\odot r}$ with $\|u_n\|\le1$
such that
$\|T_{x_n,r}u_n\|\ge\|T_{x_n,r}\|_{\op}-n^{-1}$.
Put $v_n=T_{x_n,r}u_n/\|T_{x_n,r}u_n\|$ when the denominator is nonzero,
and $v_n=0$ otherwise. Then
\begin{equation}
 \langle x_n,\Sym_q(u_n\otimes v_n)\rangle
 =\langle T_{x_n,r}u_n,v_n\rangle
 \ge\|T_{x_n,r}\|_{\op}-n^{-1}.
 \label{str:maximizing-tests}
\end{equation}
All added sequences are bounded, by
\eqref{pre:contraction-bound}. Thus
$\mathcal B=\bigcup_{j\ge0}\mathcal B_j$ is countable and closed under
the specified operations and choices of test sequences.

There are countably many contraction diagrams on finite lists from
$\mathcal B$. Include all scalar products between their outputs of
the same order, allowing permutations of output slots, among the scalar
sequences to be selected. These sequences are bounded by the classical
diagram inequality \cite[Theorem~4.1(A)]{MajorLectures}.
A diagonal subsequence makes all these scalar products, all scalar
products between same-order members of $\mathcal B$, and all
$\|T_{x_n,r}\|_{\op}$ for $x\in\mathcal B$ converge.

On the real span of the order-$q$ sequences, put
$\langle x,y\rangle_\infty=\lim_n\langle x_n,y_n\rangle$.
For a finite linear combination $z$, one has
$\langle z,z\rangle_\infty=\lim_n\|z_n\|^2\ge0$.
Quotient by the null space and complete to obtain $\cH_q$.
The classes of the rational combinations are dense, so $\cH_q$ is
separable. For a bilinear operation $B$ satisfying
$\|B(x,y)\|\le C\|x\|\|y\|$, the inequality
\[
 \|B(x,y)-B(x',y')\|
 \le C\bigl(\|x-x'\|\|y\|+\|x'\|\|y-y'\|\bigr)
\]
shows that $B$ respects the null spaces and extends to the completions.
This applies to every operation in the statement; for $\diamond$,
$C=\sqrt{\binom{p+q}{p}}$, giving \eqref{str:product-bound}.

Fix $\xi\in\cH_q$ and choose $\xi^{(j)}=[x^{(j)}]$ from the dense
family with $\|\xi^{(j)}-\xi\|\le2^{-j}$. Choose strictly increasing
$N_j$ so that, for $n\ge N_j$ and $a,b\le j$,
\[
 \bigl|\langle x_n^{(a)},x_n^{(b)}\rangle
       -\langle\xi^{(a)},\xi^{(b)}\rangle\bigr|\le j^{-1}.
\]
Set $x_n=x_n^{(j)}$ for $N_j\le n<N_{j+1}$ and $x_n=0$ for $n<N_1$.
On this interval,
\[
 \|x_n\|^2\le(\|\xi\|+1)^2+1,
 \qquad
 \left|\|x_n-x_n^{(a)}\|^2-\|\xi^{(j)}-\xi^{(a)}\|^2\right|
 \le\frac4j\quad(a\le j).
\]
Consequently
$\lim_n\|x_n-x_n^{(a)}\|=\|\xi-\xi^{(a)}\|$, so $x$ is a bounded
representative. If $y$ satisfies the same representative condition,
\[
 \limsup_n\|x_n-y_n\|
 \le\limsup_n\|x_n-x_n^{(a)}\|
    +\limsup_n\|y_n-x_n^{(a)}\|\longrightarrow0.
\]
For another dense approximation $\eta^{(b)}\to\xi$, the bound
\[
 \limsup_n\|x_n-y_n^{(b)}\|
 \le\limsup_n\|x_n-x_n^{(a)}\|+\|\xi^{(a)}-\eta^{(b)}\|
\]
proves independence of the approximation. Applying Cauchy--Schwarz
to dense approximations of two representatives proves
\eqref{str:limit-inner-product}. The bilinear estimate above proves
that their operations represent the corresponding completed elements.

Finally, put $L_a=\lim_n\|T_{x_n^{(a)},r}\|_{\op}$.
Since $\|T_{z,r}\|_{\op}\le\|z\|$,
\[
 |L_a-L_b|\le\|\xi^{(a)}-\xi^{(b)}\|,\qquad
 \limsup_n\bigl|\|T_{x_n,r}\|_{\op}-L_a\bigr|
 \le\limsup_n\|x_n-x_n^{(a)}\|.
\]
Thus $L_a$ converges, and $\|T_{x_n,r}\|_{\op}$ converges to the same
limit. Additional countably many bounded sequences can be included
in $\mathcal B_0$.
\end{proof}

\begin{proof}[Proof of Theorem~\ref{str:contraction-criterion}]
For $\eta=[u]\in\cH_r$ and $\zeta=[v]\in\cH_{q-r}$,
\begin{equation}
 \langle\xi,\eta\diamond\zeta\rangle
 =\sqrt{\binom qr}\lim_n\langle T_{x_n,r}u_n,v_n\rangle.
 \label{str:duality-test}
\end{equation}
Hence $\|T_{x_n,r}\|_{\op}\to0$ for every proper $r$ implies
$\xi\perp\cD_q$.
Conversely, suppose $\xi\perp\cD_q$. For a dense approximation
$\xi^{(a)}=[x^{(a)}]$, let $u^{(a)},v^{(a)}$ be its test sequences
from \eqref{str:maximizing-tests}, and put
$\gamma_a=[\Sym_q(u_n^{(a)}\otimes v_n^{(a)})]$.
Then $\gamma_a\in\cD_q$ and $\|\gamma_a\|\le1$. Thus
\[
 \lim_n\|T_{x_n^{(a)},r}\|_{\op}
 =\langle\xi^{(a)},\gamma_a\rangle
 =\langle\xi^{(a)}-\xi,\gamma_a\rangle
 \le\|\xi^{(a)}-\xi\|.
\]
The Lipschitz estimate for the operator norm now gives
$\limsup_n\|T_{x_n,r}\|_{\op}\le2\|\xi^{(a)}-\xi\|\to0$.
Moreover, the class $[x_n\otimes_ru_n]\in\cH_{q-r}$ is zero if and
only if its inner product with every $\zeta\in\cH_{q-r}$ is zero.
Identity \eqref{str:duality-test} therefore gives the equivalent
condition involving all contraction tests.

Regard $x_n$ as an element of
$H_n^{\odot(q-r)}\otimes H_n^{\odot r}$. The associated operator
$T=T_{x_n,r}$ satisfies
$\|T\|_{\HS}^2=\|x_n\|^2$, and $TT^*$ has kernel
$x_n\otimes_rx_n$. For its singular values $(\sigma_j)$,
\[
 \|x_n\otimes_rx_n\|^2=\sum_j\sigma_j^4,
 \qquad \|x_n\|^2=\sum_j\sigma_j^2.
\]
Consequently
\begin{equation}
 \|T_{x_n,r}\|_{\op}^4
 \le\|x_n\otimes_rx_n\|^2
 \le\|T_{x_n,r}\|_{\op}^2\|x_n\|^2.
 \label{str:singular-value-estimate}
\end{equation}
The representative is bounded, proving the self-contraction
characterization. Finally, for $r<p$, apply $T_{x_n,r}$ to the first
$r$ slots of $y_n$; the other slots carry the identity operator. This
gives
\[
 \|x_n\otimes_ry_n\|
 \le\|T_{x_n,r}\|_{\op}\|y_n\|\longrightarrow0.
\]
If $r=p<q$, apply the same argument with the inputs exchanged.
The remaining case is $p=q=r$, a scalar inner product.
\end{proof}

\begin{lemma}[Stability under enlargement]
\label{str:enlargement}
An enlargement of the finite-degree bound and of the family by
countably many sequences, followed by a further subsequence, embeds
the old spaces isometrically in the new ones and preserves their
decomposable and primitive parts.
\end{lemma}
\begin{proof}
Old scalar-product limits are unchanged. Products preserve
$\cD_q\subset\cD'_q$, whereas
Theorem~\ref{str:contraction-criterion} preserves
$\cG_q\subset\cG'_q$. Thus an old orthogonal decomposition
is also its decomposition in the enlarged space.
\end{proof}

\subsection{Gaussian limits and pairings of primitive factors}

\begin{lemma}[Joint Gaussian convergence of primitive coordinates]
\label{str:joint-normal}
For $h_a=[h_{a,n}]\in\cG_{s_a}$, $1\le a\le d$,
\[
 \bigl(I_{s_a}^{W_n}(h_{a,n}/\sqrt{s_a!})\bigr)_{a\le d}
 \Longrightarrow N_d(0,C),\qquad
 C_{ab}=\mathbf1_{\{s_a=s_b\}}\langle h_a,h_b\rangle.
\]
\end{lemma}
\begin{proof}
Put $X_{n,a}=I_{s_a}(h_{a,n}/\sqrt{s_a!})$.
If $\|h_a\|=0$, then $\E X_{n,a}^2=\|h_{a,n}\|^2\to0$.
Otherwise Theorem~\ref{str:contraction-criterion} gives every proper
self-contraction tending to zero, and the fourth-moment theorem
\cite{NualartPeccati} yields
$X_{n,a}\weak N(0,\|h_a\|^2)$; coordinates of order one are Gaussian
already. Orthogonality of Wiener chaoses and
\eqref{str:limit-inner-product} give
\[
 \E[X_{n,a}X_{n,b}]
 =\1_{\{s_a=s_b\}}\langle h_{a,n},h_{b,n}\rangle
 \longrightarrow C_{ab}.
\]
If $C$ is singular, diagonalize its blocks of equal chaos order
by orthogonal changes of coordinates. Such changes preserve each
homogeneous order; the zero-variance coordinates converge to zero in
$L^2$. The remaining limiting covariance is positive definite, and
the multivariate theorem \cite{PeccatiTudor} gives their joint Gaussian
limit. Restoring the zero coordinates proves the statement.
\end{proof}

\begin{lemma}[Diagram estimate and primitive pairing identity]
\label{str:diagram-bound}\label{str:pairings}
For every diagram of Definition~\ref{str:diagram},
\begin{equation}
 \|\mathcal C_{\mathfrak d}((z_v)_v)\|
 \le\prod_v\|z_v\|.
 \label{str:diagram-estimate}
\end{equation}
If exactly $r$ edges join vertices $a,b$, the bound improves to
$\|z_a\otimes_r z_b\|\prod_{v\ne a,b}\|z_v\|$.
For primitive factors $h_i\in\cG_{s_i}$, $k_j\in\cG_{t_j}$ with
$\sum_i s_i=\sum_jt_j=q$,
\begin{equation}
 \left\langle\mathop{\diamond}_{i=1}^m h_i,
                   \mathop{\diamond}_{j=1}^l k_j\right\rangle
 =\mathbf1_{\{m=l\}}
 \sum_{\substack{\pi\in\mathfrak S_m\\s_i=t_{\pi(i)}}}
 \prod_i\langle h_i,k_{\pi(i)}\rangle.
 \label{str:pairing-formula}
\end{equation}
\end{lemma}
\begin{proof}
The norm inequality is the classical diagram bound
\cite[Theorem~4.1(A)]{MajorLectures}. To obtain the refined bound,
contract all $r$ edges between $a$ and $b$ first. Merge these vertices
into the tensor $z_a\otimes_rz_b$, permuting its free slots as required.
No self-contraction remains, so applying the first bound to the
resulting diagram proves the assertion.

For the pairing identity, set $H_n=\bigotimes_{i=1}^m h_{i,n}$,
$K_n=\bigotimes_{j=1}^l k_{j,n}$,
$S=\prod_i s_i!$ and $T=\prod_jt_j!$.
Orthogonal symmetrization gives the exact identity
\[
 \left\langle\mathop{\diamond}_i h_{i,n},
                    \mathop{\diamond}_j k_{j,n}\right\rangle
 =\frac1{\sqrt{ST}}\sum_{\tau\in\mathfrak S_q}
                  \langle H_n,P_\tau K_n\rangle.
\]
In a summand not pairing whole factors, some pair $(h_{i,n},k_{j,n})$
is joined by $r\ge1$ slots with $(s_i,t_j,r)\ne(s_i,s_i,s_i)$.
The refined bound makes its absolute value at most
\[
 \|h_{i,n}\otimes_rk_{j,n}\|
       \prod_{a\ne i}\|h_{a,n}\|\prod_{b\ne j}\|k_{b,n}\|,
\]
which tends to zero by \eqref{str:cross-contraction}.
A surviving summand requires $m=l$ and a matching $\pi$ with
$s_i=t_{\pi(i)}$. For each such matching, exactly $S$ slot
permutations contribute, each giving
$\prod_i\langle h_{i,n},k_{\pi(i),n}\rangle$.
Since $S=T$, their coefficient is $S/\sqrt{ST}=1$.
There are finitely many permutations; taking limits proves
\eqref{str:pairing-formula}.
\end{proof}

\begin{proof}[Proof of Theorem~\ref{str:representation}]
The Wick covariance formula \cite[Chapter~III]{Janson} agrees with
\eqref{str:pairing-formula}. Hence the proposed map preserves inner
products on finite linear combinations of Wick monomials. In
particular, every zero polynomial has zero image; the map is
well defined and extends to an isometry $\mathbb U_q$ with closed
range.

For $q=1$, its range is $\cG_1=\cH_1$. Suppose it is onto for all
orders below $q$. For complementary orders $r,q-r$, approximate
$\xi\in\cH_r$ and $\eta\in\cH_{q-r}$ by the images of algebraic
Wick polynomials $P_j,Q_j$. Concatenating their Wick factors gives
an algebraic polynomial of weighted degree $q$ whose image is
$\mathbb U_rP_j\diamond\mathbb U_{q-r}Q_j$. By
\eqref{str:product-bound}, these images converge to
$\xi\diamond\eta$. The closed range therefore contains $\cD_q$;
it also contains $\cG_q$, by the one-factor definition. It equals
$\cH_q$. The one-factor space maps onto $\cG_q$, and every product
with at least two factors belongs to $\cD_q$. Isometry identifies
the stated orthogonal subspaces.

For laws, fix finitely many primitive factors and write
\[
 X_{i,n}=I_{s_i}(h_{i,n}/\sqrt{s_i!}),\qquad
 C_{ij,n}=\1_{\{s_i=s_j\}}\langle h_{i,n},h_{j,n}\rangle.
\]
For a subset $A$ of the factor labels, put
\[
 Y_{A,n}=I_{\sum_{i\in A}s_i}
 \left(\frac{\mathop{\diamond}_{i\in A}h_{i,n}}
 {\sqrt{(\sum_{i\in A}s_i)!}}\right),\qquad
 Y_{\varnothing,n}=1.
\]
The product diagram formula \cite[Theorem~4.1(B)]{MajorLectures} gives
\begin{equation}
 \prod_{i\in A}X_{i,n}
 =\sum_M\left(\prod_{\{i,j\}\in M}C_{ij,n}\right)
       Y_{A\setminus V(M),n}+\epsilon_{A,n},
 \qquad \|\epsilon_{A,n}\|_2\longrightarrow0,
 \label{str:product-matchings}
\end{equation}
where $M$ ranges over matchings of labels in $A$.
Indeed, a discarded diagram contains a contraction from
\eqref{str:cross-contraction}; its tensor norm tends to zero by
Lemma~\ref{str:diagram-bound}, hence its integral tends to zero in
$L^2$ by the Wiener isometry. There are finitely many diagrams.
A full pairing of factors of order $s$ has $s!$ slot bijections,
which cancel their normalizing factors $\sqrt{s!}\sqrt{s!}$.
The unpaired factors give $Y_{A\setminus V(M),n}$.

The inverse of these finite identities is
\[
 Y_{A,n}=\sum_M(-1)^{|M|}
       \left(\prod_{\{i,j\}\in M}C_{ij,n}\right)
       \prod_{i\in A\setminus V(M)}X_{i,n}+o_{L^2}(1).
\]
To check it, substitute \eqref{str:product-matchings} in the right
side. The coefficient of a fixed total matching $L$ is
$\sum_{M\subseteq L}(-1)^{|M|}=(1-1)^{|L|}$, so only the empty
matching remains. All error terms tend to zero in $L^2$, since the
finitely many covariance coefficients are bounded. The displayed
polynomial is the Wick polynomial with covariance $C_n$.
By Lemma~\ref{str:joint-normal}, $X_n$ converges jointly to its Gaussian
counterpart and $C_n\to C$. The continuous mapping theorem proves
joint convergence for every finite collection of algebraic inputs.

For general inputs, choose algebraic $P_{a,q}^{(j)}$ with
$\delta_j^2=\sum_{a,q}\|\xi_{a,q}-\mathbb U_qP_{a,q}^{(j)}\|^2\to0$.
Let $F_n^{(j)}$ be their representative integrals and let
$F^{(j)}=(\sum_qP_{a,q}^{(j)})_a$. Orthogonality of orders gives
\[
 \lim_n\E|F_n-F_n^{(j)}|^2
 =\E|F-F^{(j)}|^2=\delta_j^2.
\]
For every bounded $1$-Lipschitz function $g$,
\[
 \limsup_n|\E g(F_n)-\E g(F)|
 \le2\delta_j+
     \limsup_n|\E g(F_n^{(j)})-\E g(F^{(j)})|=2\delta_j.
\]
Let $j\to\infty$. Lemma~\ref{pre:tightness} gives convergence of
all mixed moments. All arguments use the same selected subsequence.
\end{proof}

\begin{proof}[Proof of Theorem~\ref{str:contractions}]
For algebraic inputs, expand the finitely many symmetrizations and
regard every primitive factor as a separate vertex. If a term is not
retained, some two factors $h_{v,i,n},h_{w,j,n}$ are joined by
$r\ge1$ edges without forming a complete pair of equal orders.
Contracting all their common edges first gives the bound
\[
 \|h_{v,i,n}\otimes_rh_{w,j,n}\|
 \prod_{(a,b)\ne(v,i),(w,j)}\|h_{a,b,n}\|\longrightarrow0.
\]
This is \eqref{str:cross-contraction} followed by
Lemma~\ref{str:diagram-bound}. In a retained term every scalar
factor $\langle h_{v,i,n},h_{w,j,n}\rangle$ converges to its
sequence-space inner product. Replacing finitely many such factors
changes the output by a tensor of norm tending to zero, since all
uncontracted factors have bounded norms. Summing proves
\eqref{str:open-contraction-limit}.

For arbitrary inputs, choose algebraic $\xi_v^{(j)}\to\xi_v$ and
representatives $x_{v,n}^{(j)}$. Multilinearity gives
\begin{align*}
 &\limsup_n\|\mathcal C_{\mathfrak d}((x_{v,n})_v)
              -\mathcal C_{\mathfrak d}((x_{v,n}^{(j)})_v)\|\\
 &\quad\le\sum_v\|\xi_v-\xi_v^{(j)}\|
       \prod_{u\ne v}\max\{\|\xi_u\|,\|\xi_u^{(j)}\|\}
       \longrightarrow0.
\end{align*}
For outputs of order $\ell$, use the Hilbert space obtained by
quotienting the selected output sequences by zero limiting norm and
completing. These spaces were included in the scalar-product selection
in Proposition~\ref{str:sequence-spaces}, even when $\ell>Q$ or the
outputs are not symmetric. The displayed estimate makes the outputs
of the approximations Cauchy, independently of the approximating
family, and defines the continuous extension. The finite-level identity
\[
 \langle\mathcal C_{\mathfrak d}(x),
         \mathcal C_{\mathfrak e}(y)\rangle
 =\mathcal C_{\mathfrak d\star\mathfrak e}(x,y)
\]
holds when the output orders agree: $\mathfrak d\star\mathfrak e$
joins corresponding output slots in disjoint copies of the two
input lists. It is a closed diagram with no self-contraction. This
identity identifies all limiting scalar products and norms.
\end{proof}

\section{Kernel extraction and Gaussian approximation}
\label{ext:algorithm}
\label{ext:section}

We construct orthogonal kernel projections for which the residual is
close to a Gaussian vector independent of the projected part. The first
procedure applies to a finite sum of chaoses. The second applies to a
homogeneous input and converges to the decomposable and primitive
projections of Section~\ref{framework:section}. The dimension bounds
refer to the selected kernel subspaces.
Related regularization methods are developed in
\cite{DeServedio,DeMosselNeeman}; the Gaussian comparison below uses
Malliavin--Stein interpolation
\cite{NourdinPeccatiBook,NourdinNualartPeccatiStable,TudorQuantitativeIndependence}.

\subsection{Contraction defects and the extraction operator}

For centered random variables in a finite sum of chaoses, write
\[
 \Gamma_{\rm MS}(F,G)=\langle DF,-DL^{-1}G\rangle.
\]
The generator satisfies $LI_q(h)=-qI_q(h)$, and
$\mathbb E\Gamma_{\rm MS}(F,G)=\mathbb E[FG]$. For $p,q\ge1$ and
$1\le s\le p\wedge q$, put
\begin{equation}
 b_{p,q,s}=\frac pq\binom{p-1}{s-1}\binom{q-1}{s-1}
                 \binom{p+q-2s}{p-s}.
 \label{ext:b}
\end{equation}
The standard Malliavin covariance expansion
\cite[Section~3]{NoreddineNourdin}, in our normalization, is
\begin{equation}
 \Gamma_{\rm MS}\left(I_p(x/\sqrt{p!}),I_q(y/\sqrt{q!})\right)
 =\sum_{s=1}^{p\wedge q}\sqrt{b_{p,q,s}}\,
 I_{p+q-2s}\left(\frac{x\wtensor_sy}{\sqrt{(p+q-2s)!}}\right).
 \label{ext:bracket}
\end{equation}
Indeed, differentiating the integrals and using
$-DL^{-1}I_q=q^{-1}DI_q$ gives the coefficient
\[
 \frac{p(s-1)!\binom{p-1}{s-1}\binom{q-1}{s-1}
             \sqrt{(p+q-2s)!}}{\sqrt{p!q!}}
 =\sqrt{b_{p,q,s}}
\]
by the product formula on orders $p-1,q-1$. Here $I_0(c)=c$;
for $p=q=s$ the coefficient is $1$ and the term is
$\langle x,y\rangle$.
For $Q\ge2$ define
\begin{equation}
 a_Q=\sum_{k=1}^{Q-1}\binom{Q-1}{k}^{2}\binom{2k}{k},
 \qquad \mu_Q=\left\lfloor\frac{Q^2+Q+1}{3}\right\rfloor,
 \label{ext:constants}
\end{equation}
and set $a_1=0$, $\mu_1=1$. In particular, $a_2=2$, $a_3=14$,
$a_4=92$.

\begin{lemma}[Coefficient and multiplicity bounds]
\label{ext:constant-identities}
For $1\le p,q\le Q$,
\[
 \sum_{\substack{1\le s\le p\wedge q\\\neg(p=q=s)}}b_{p,q,s}
 \le a_Q.
\]
At each positive output order $\ell$, at most $\mu_Q$ triples
$(p,q,s)$ satisfy $p+q-2s=\ell$. Both bounds are attained when
$Q\ge2$. For $Q=1$, $\mu_1=1$ is only a bounding convention;
there is no positive output order.
\end{lemma}
\begin{proof}
For $p\ge q$, comparison of the binomial factors gives
$b_{p,q,s}\le b_{p,p,s}$. Explicitly, with $i=p-s\ge j=q-s$,
\[
 \frac{\binom{2i}{i}}{\binom{i+j}{i}}
 =\prod_{v=j+1}^{i}\frac{i+v}{v}\ge\frac pq
 \quad(i>j),
\]
where the last inequality follows from
$(i+j+1)(s+j)-(j+1)(s+i)=i(s-1)+j(j+1)\ge0$.
For $p<q$, use $b_{p,q,s}=(p/q)^2b_{q,p,s}$.
The diagonal sums increase with the order and their value at $Q$ is
$a_Q$. The number of triples at order $\ell$ is
\[
 m_Q(\ell)=\sum_{\substack{i+j=\ell\\0\le i,j\le Q-1}}
              (Q-\max(i,j)).
\]
For $1\le\ell\le Q-1$ its even and odd values are
\[
 m_Q(2t)=(2t+1)Q-3t^2-2t,\qquad
 m_Q(2t+1)=(2t+2)Q-3t^2-5t-2.
\]
For $\ell\ge Q-1$, putting $h=2Q-2-\ell$ gives
$m_Q(\ell)=\sum_{u=0}^{h}(1+\min(u,h-u))$, which decreases with
$\ell$. On the first range, completing the squares gives
\begin{align*}
 m_Q(2t)&=\frac{Q^2+Q+1}{3}
          -3\left(t-\frac{Q-1}{3}\right)^2,\\
 m_Q(2t+1)&=\frac{Q^2+Q+1}{3}-\frac14
          -3\left(t-\frac{2Q-5}{6}\right)^2.
\end{align*}
Since $m_Q(\ell)$ is integer-valued, these identities give the
upper bound $\mu_Q$. For $Q\ge2$ it is attained at
$\ell=\lfloor2Q/3\rfloor$: for $Q=3k$, $3k+1$ and $3k+2$,
the respective values are $3k^2+k$, $3k^2+3k+1$ and
$3k^2+5k+2$. The decreasing second range gives no larger value.
\end{proof}

Fix rescaled kernels $x_{a,q}\in H^{\odot q}$, $a\le d$, $q\le Q$,
with
\begin{equation}
 F_a=\sum_{q=1}^QI_q(x_{a,q}/\sqrt{q!}),\qquad
 \sum_{a,q}\|x_{a,q}\|^2\le R^2.
 \label{ext:normalization}
\end{equation}
\begin{definition}[Active and residual parts]
\label{str:active-residual}
For a closed subspace $M_q^{[J]}\subset H^{\odot q}$ and its
orthogonal projection $P_{M_q^{[J]}}$, put
$x_{a,q}^{\mathrm{act},J}=P_{M_q^{[J]}}x_{a,q}$ and
$x_{a,q}^{\mathrm{res},J}=x_{a,q}-x_{a,q}^{\mathrm{act},J}$.
The associated vectors have coordinates
\[
 F_a^{\mathrm{act},J}=\sum_q
 I_q(x_{a,q}^{\mathrm{act},J}/\sqrt{q!}),\qquad
 F_a^{\mathrm{res},J}=\sum_q
 I_q(x_{a,q}^{\mathrm{res},J}/\sqrt{q!}).
\]
The procedures below specify $M_q^{[J]}$. When kernels depend on
$n$, the corresponding subspaces are denoted by $M_{n,q}^{[J]}$.
\end{definition}

The subspaces are common to all components. Until the procedures are
defined, we suppress the level index $J$.

\begin{definition}[Malliavin--Stein defects]
\label{ext:parts}
Writing $C=\operatorname{Cov}(F^{\mathrm{res}})$, the residual and
mixed defects are
\begin{align}
 D_R&=\sum_{a,b}\|\Gamma_{\rm MS}(F^{\mathrm{res}}_a,
                       F^{\mathrm{res}}_b)-C_{ab}\|_2^2,
 \label{ext:residual-defect}\\
 D_I&=\sum_{a,b}\|\Gamma_{\rm MS}(F^{\mathrm{act}}_a,
                       F^{\mathrm{res}}_b)\|_2^2.
 \label{ext:mixed-defect}
\end{align}
In particular, $\mathbb E[F^{\mathrm{act}}_aF^{\mathrm{res}}_b]=0$ for
all $a,b$.
\end{definition}

For brevity within this section write $h_{a,q}=x^{\mathrm{act}}_{a,q}$
and $r_{a,q}=x^{\mathrm{res}}_{a,q}$. We keep the contractions separate
in the auxiliary nonnegative quantity
\begin{equation}
 \mathcal S(h,r)=\sum_{a,b,p,q}
 \sum_{s=1}^{p\wedge q}{}^{\prime} b_{p,q,s}
 \left(\|r_{a,p}\wtensor_sr_{b,q}\|^2+
       \|h_{a,p}\wtensor_sr_{b,q}\|^2\right),
 \label{ext:total-sum}
\end{equation}
where the prime omits $p=q=s$. The omitted mixed scalar terms vanish.
For fixed $a,b$ and a positive output order $\ell$, the corresponding
rescaled kernel of a bracket is a sum
$\sum_{p+q-2s=\ell}\sqrt{b_{p,q,s}}\,z_{a,p}\wtensor_sr_{b,q}$,
where $z=r$ or $z=h$. Its squared norm is at most
\[
 m_Q(\ell)\sum_{p+q-2s=\ell}
 b_{p,q,s}\|z_{a,p}\wtensor_sr_{b,q}\|^2.
\]
Summing over $\ell,a,b$ and using the Wiener isometry gives
\begin{equation}
 D_R+D_I\le\mu_Q\mathcal S(h,r).
 \label{ext:comparison}
\end{equation}
If the input is homogeneous of order $q$, every output order occurs
only once, and equality holds with factor $1$.

Let $\Lambda_q$ map $H^{\odot q}$ into the Hilbert direct sum of the
indicated tensor spaces by
\begin{equation}
 \Lambda_q u=
 \left(\bigl(\sqrt{b_{p,q,s}}\,r_{a,p}\wtensor_sP_{M_q^\perp}u\bigr)',
       \bigl(\sqrt{b_{p,q,s}}\,h_{a,p}\wtensor_sP_{M_q^\perp}u\bigr)'
 \right)_{a,p,s},\qquad
 \mathsf B_q=\Lambda_q^*\Lambda_q.
 \label{ext:Lambda}
\end{equation}
For every $u\in H^{\odot q}$, the contraction bound gives
\[
 \|\Lambda_qu\|^2
 \le\|P_{M_q^\perp}u\|^2
   \sum_{a,p}(\|r_{a,p}\|^2+\|h_{a,p}\|^2)
                    \sum_s{}'b_{p,q,s}
 \le a_QR^2\|u\|^2.
\]
Since $\Lambda_q=\Lambda_qP_{M_q^\perp}$ and
$P_{M_q^\perp}r_{b,q}=r_{b,q}$, this proves
\begin{equation}
 0\le\mathsf B_q\le a_QR^2\operatorname{Id},\qquad
 \mathsf B_q=P_{M_q^\perp}\mathsf B_qP_{M_q^\perp},\qquad
 \mathcal S(h,r)=\sum_{b,q}\langle r_{b,q},\mathsf B_qr_{b,q}\rangle.
 \label{ext:positive-identities}
\end{equation}

\subsection{Finite-degree extraction with an explicit defect bound}

\begin{algorithm}[Simultaneous contraction extraction]
\label{ext:simultaneous-algorithm}
Start from $M_q=\{0\}$. At level $J=1,2,\ldots$, continue the preceding
construction while $\mathcal S(h,r)>a_QR^4/J$. At each operation set
\[
 V_q=\operatorname{span}\{\mathsf B_qr_{b,q}:1\le b\le d\},\qquad
 M_q\leftarrow M_q\oplus V_q\quad(1\le q\le Q),
\]
and recompute the projections and operators. Denote the subspaces at
termination of level $J$ by $M_q^{[J]}$.
\end{algorithm}

\begin{theorem}[Finite-degree extraction]
\label{ext:finite}
Assume \eqref{ext:normalization}, $Q\ge2$, and $R>0$.
Procedure~\ref{ext:simultaneous-algorithm} terminates at every level.
Its subspaces are nested, and
\begin{equation}
 \dim M_q^{[J]}\le d(J-1),\qquad
 \mathcal S(h^J,r^J)\le\frac{a_QR^4}{J},\qquad
 D_R+D_I\le\frac{\mu_Qa_QR^4}{J}.
 \label{ext:finite-bounds}
\end{equation}
There are at most $J-1$ operations in total up to level $J$.
For $Q=1$, the choice $M_1^{[J]}=\{0\}$ has zero defects.
\end{theorem}
\begin{proof}
Put $K=a_QR^2>0$. At an operation, $V_q\subset M_q^\perp$, so
$P_{M_q\oplus V_q}=P_{M_q}+P_{V_q}$ and the new residual is
$r'_{b,q}=r_{b,q}-P_{V_q}r_{b,q}$. Thus its total energy decreases by
\[
 \sum_{b,q}(\|r_{b,q}\|^2-\|r'_{b,q}\|^2)
 =\sum_{b,q}\|P_{V_q}r_{b,q}\|^2.
\]
For $0\le\mathsf B\le K\Id$, one has $\mathsf B^2\le K\mathsf B$.
If $\mathsf Bz\ne0$, then
\[
 \|P_{\R\mathsf Bz}z\|^2
 =\frac{\langle z,\mathsf Bz\rangle^2}{\|\mathsf Bz\|^2}
 \ge K^{-1}\langle z,\mathsf Bz\rangle.
\]
When $\mathsf Bz=0$ both sides of the resulting inequality are zero.
Since $\mathsf B_qr_{b,q}\in V_q$, the energy decrease is at least
\[
 K^{-1}\sum_{b,q}\langle r_{b,q},\mathsf B_qr_{b,q}\rangle
 =\frac{\mathcal S(h,r)}{a_QR^2}.
\]
An operation at level $j\le J$ therefore removes strictly more than
$R^2/j\ge R^2/J$. The initial energy is at most $R^2$; $J$ such
operations would make it negative. Thus there are at most $J-1$
operations up to level $J$, and every level terminates. Each operation
adds at most $d$ dimensions to each subspace. At termination the
stopping inequality and \eqref{ext:comparison} give
\eqref{ext:finite-bounds}. For $Q=1$, all brackets are deterministic
covariances, so both defects are zero without extraction.
\end{proof}

\subsection{Distances to an independent Gaussian residual}

For laws on $\mathbb R^m$, let $d_2$ be the supremum of expectation
differences over $C^2$ functions with
$\sup\|\nabla h\|\le1$ and $\sup\|\nabla^2h\|_{\rm op}\le1$.
Let $W_1$ denote the Wasserstein distance for the Euclidean norm.

\begin{lemma}[Gaussian replacement with an arbitrary retained vector]
\label{ext:stein}
Let $A,B$ be centered $\mathbb R^d$-valued vectors in a finite sum of
Wiener chaoses. Define $D_R,D_I$ as in
\eqref{ext:residual-defect}--\eqref{ext:mixed-defect}, with
$C=\operatorname{Cov}(B)$. If $G_C\sim N_d(0,C)$ is independent of
$A$, then, with $\Delta=D_R+D_I$,
\begin{align}
 d_2\bigl((A,B),(A,G_C)\bigr)
 &\le\sqrt d\left(\sqrt{D_I}+\tfrac12\sqrt{D_R}\right)
 \le\frac{\sqrt{5d}}2\,\Delta^{1/2},
 \label{ext:smooth-bound}\\
 W_1\bigl((A,B),(A,G_C)\bigr)
 &\le2\,10^{1/4}\sqrt d\,\Delta^{1/4}.
 \label{ext:wasserstein-bound}
\end{align}
No invertibility assumption on $C$ is required.
\end{lemma}
\begin{proof}[Proof by Gaussian interpolation]
The interpolation follows the method of
\cite[Proposition~3.1]{NourdinNualartPeccatiStable} and
\cite{TudorQuantitativeIndependence}. Let $Z\sim N_d(0,I_d)$ be
independent of the original Gaussian field and put $G_C=C^{1/2}Z$.
For a $C^2$ function $h$ with bounded gradient and
$\sup\|\nabla^2h\|_{\op}\le H$, set
$B_t=\sqrt tB+\sqrt{1-t}G_C$ and $\psi(t)=\E h(A,B_t)$.
For $0<t<1$, differentiation under the expectation gives
\[
 \psi'(t)=\frac1{2\sqrt t}\sum_j\E[B_j\partial_{b_j}h(A,B_t)]
 -\frac1{2\sqrt{1-t}}\sum_j\E[(G_C)_j\partial_{b_j}h(A,B_t)].
\]
Condition on $G_C$ in the first term and use \eqref{pre:ibp} with
$Y=B_j$. The chain rule gives
\begin{align*}
 \E[B_j\partial_{b_j}h(A,B_t)]
 &=\sum_\ell\E[\partial_{a_\ell b_j}h(A,B_t)
                    \Gamma_{\rm MS}(A_\ell,B_j)]\\
 &\quad+\sqrt t\sum_i\E[\partial_{b_i b_j}h(A,B_t)
                    \Gamma_{\rm MS}(B_i,B_j)].
\end{align*}
Gaussian integration by parts in $Z$ gives
\[
 \E[(G_C)_j\partial_{b_j}h(A,B_t)]
 =\sqrt{1-t}\sum_i C_{ij}\E[\partial_{b_i b_j}h(A,B_t)].
\]
No inverse of $C$ is used. Combining these identities yields
\begin{align*}
 \psi'(t)
 &=\frac12\sum_{i,j}\E[\partial_{b_ib_j}h(A,B_t)
                    (\Gamma_{\rm MS}(B_i,B_j)-C_{ij})]\\
 &\quad+\frac1{2\sqrt t}\sum_{\ell,j}\E[
          \partial_{a_\ell b_j}h(A,B_t)\Gamma_{\rm MS}(A_\ell,B_j)].
\end{align*}
Each $d\times d$ derivative block has Hilbert--Schmidt norm at most
$\sqrt dH$. Cauchy--Schwarz therefore gives the integrable bound
\[
 |\psi'(t)|\le\frac{H\sqrt d}{2}
                    \left(\sqrt{D_R}+t^{-1/2}\sqrt{D_I}\right).
\]
Integrate on $[\delta,1-\delta]$ and let $\delta\downarrow0$.
The bounded gradient and $B_t\to B_0,B_1$ in $L^2$ justify the endpoint
limits. Since $\int_0^1t^{-1/2}\,dt=2$,
\[
 |\psi(1)-\psi(0)|
 \le H\sqrt d(\sqrt{D_I}+\tfrac12\sqrt{D_R})
 \le\frac{H\sqrt{5d}}2\Delta^{1/2}.
\]
This proves \eqref{ext:smooth-bound}.

For a $1$-Lipschitz $g:\R^{2d}\to\R$, let
$g_\eps(x)=\E g(x+\eps Z')$, $Z'\sim N_{2d}(0,I_{2d})$.
Then $|g_\eps(x)-g(x)|\le\eps\sqrt{2d}$ and
$\|\nabla g_\eps\|\le1$. For unit vectors $u,v$, differentiating
the Gaussian convolution in the weak-derivative sense gives
\[
 \partial_v\partial_u g_\eps(x)
 =\eps^{-1}\E[\partial_u g(x+\eps Z')\langle Z',v\rangle],
 \qquad |\partial_v\partial_u g_\eps(x)|\le\eps^{-1}.
\]
Apply the preceding estimate to $g_\eps$. The expectation difference
for $g$ is at most
\[
 2\eps\sqrt{2d}+\frac{\sqrt{5d}}{2\eps}\Delta^{1/2}.
\]
For $\Delta>0$, minimizing $a\eps+b/\eps$ gives
$2\sqrt{ab}=2\,10^{1/4}\sqrt d\,\Delta^{1/4}$.
For $\Delta=0$, let $\eps\downarrow0$.
The Kantorovich dual formula gives \eqref{ext:wasserstein-bound}.
\end{proof}

\begin{corollary}[Finite-level replacement rates]
\label{ext:rates}
Under Theorem~\ref{ext:finite}, let
$C^J=\operatorname{Cov}(F^{\mathrm{res},J})$ and let $G_{C^J}$ be
independent of $F^{\mathrm{act},J}$. Then
\begin{align*}
 d_2\bigl((F^{\mathrm{act},J},F^{\mathrm{res},J}),
          (F^{\mathrm{act},J},G_{C^J})\bigr)
 &\le\frac{\sqrt{5d\mu_Qa_Q}}2\frac{R^2}{\sqrt J},\\
 W_1\bigl((F^{\mathrm{act},J},F^{\mathrm{res},J}),
          (F^{\mathrm{act},J},G_{C^J})\bigr)
 &\le2\,10^{1/4}\sqrt d\,(\mu_Qa_Q)^{1/4}\frac R{J^{1/4}}.
\end{align*}
\end{corollary}

\subsection{Spectral selection for a homogeneous input}

We now assume that all input kernels have the same order $q\ge2$.
Write $M=M_q$, $\mathsf B=\mathsf B_q$, and define the column operator
and its Gram matrix by
\[
 T:\mathbb R^d\longrightarrow H^{\odot q},\quad
 Tc=\sum_b c_br_b,\qquad
 \mathsf K=T^*\mathsf BT.
\]
Then $\mathsf K\succeq0$, $\|T\|_{\rm op}\le R$, and
$\operatorname{Tr}\mathsf K=\mathcal S(h,r)$.

\begin{algorithm}[Spectral contraction extraction]
\label{ext:spectral-algorithm}
Start from $M=\{0\}$. At level $J$, continue the preceding construction
while $\mathcal S(h,r)>d a_qR^4/J$. Let $\lambda$ be the largest
eigenvalue of $\mathsf K$ and $V^{\rm top}\subset\mathbb R^d$ its full
eigenspace. Set
\[
 V=\mathsf BT(V^{\rm top}),\qquad M\leftarrow M\oplus V,
\]
and recompute the projections and operators. Denote the resulting
subspace by $M^{[J]}$.
\end{algorithm}

\begin{theorem}[Spectral extraction and a fixed-level gap]
\label{ext:reduced}
For a homogeneous input satisfying $\sum_a\|x_a\|^2\le R^2$,
Procedure~\ref{ext:spectral-algorithm} terminates and gives nested
subspaces satisfying
\begin{equation}
 \dim M^{[J]}\le d(J-1),\qquad
 D_R+D_I=\mathcal S(h^J,r^J)\le\frac{d a_qR^4}{J}.
 \label{ext:pure-reduced-bounds}
\end{equation}
At every operation performed up to level $J$,
\begin{equation}
 \|\mathsf BTc\|>\frac{a_qR^3}{J}\|c\|
 \quad(c\in V^{\rm top}\setminus\{0\}).
 \label{ext:fixed-level-gap}
\end{equation}
The construction is unchanged by orthogonal changes of the vector
components. The replacement estimates of Lemma~\ref{ext:stein} apply
with $\Delta=d a_qR^4/J$.
\end{theorem}
\begin{proof}
If $R=0$, all kernels are zero and no operation is performed. Suppose
$R>0$. At an operation on level $j\le J$,
\[
 \lambda\ge\frac{\Tr\mathsf K}{d}
 =\frac{\mathcal S(h,r)}d>\frac{a_qR^4}{j}
 \ge\frac{a_qR^4}{J}.
\]
For $c\in V^{\rm top}$,
\[
 \lambda\|c\|^2=\langle Tc,\mathsf BTc\rangle
 \le R\|c\|\|\mathsf BTc\|.
\]
This proves \eqref{ext:fixed-level-gap}, so
$\mathsf BT$ is injective on $V^{\rm top}$.
For a unit vector $c\in V^{\rm top}$, the line
$\R\mathsf BTc$ is contained in $V$. Using
$\mathsf B^2\le a_qR^2\mathsf B$,
\[
 \sum_b\|P_Vr_b\|^2=\|P_VT\|_{\HS}^2
 \ge\|P_VTc\|^2
 \ge\frac{\langle Tc,\mathsf BTc\rangle^2}{\|\mathsf BTc\|^2}
 \ge\frac{\lambda}{a_qR^2}
 \ge\frac{\mathcal S(h,r)}{d a_qR^2}.
\]
This is the decrease in residual energy because $V\subset M^\perp$.
It exceeds $R^2/J$ for each operation up to level $J$.
The same energy count as in Theorem~\ref{ext:finite} gives at most
$J-1$ operations, each adding $\dim V=\dim V^{\rm top}\le d$.
The stopping condition proves \eqref{ext:pure-reduced-bounds}.

Under an orthogonal change $x'=Ox$ of the $d$ inputs, induction on
operations gives the same current $M$. Residuals and active parts
are multiplied by $O$, and their index occurs in a Hilbert direct
sum in $\Lambda$; hence $\mathsf B'=\mathsf B$. Also
\[
 T'=TO^{\mathsf T},\qquad
 \mathsf K'=O\mathsf KO^{\mathsf T},\qquad
 (V^{\rm top})'=OV^{\rm top}.
\]
Thus $V'=\mathsf BTO^{\mathsf T}(OV^{\rm top})=V$ and the traces,
stopping decisions and all selected subspaces agree.
\end{proof}

\subsection{Identification of the limiting projections}

Consider rescaled kernels $x_{n,a}\in H_n^{\odot q}$ with
$\sup_n\sum_a\|x_{n,a}\|^2\le R^2$, and put
$F_{n,a}=I_q^{W_n}(x_{n,a}/\sqrt{q!})$.
Apply Procedure~\ref{ext:spectral-algorithm} at every $n$.
Before constructing the sequence spaces of
Proposition~\ref{str:sequence-spaces}, include the kernels $x_{n,a}$,
the projections at every integer level, and orthonormal bases of all
subspaces selected at each operation. These form countably many bounded
sequences. Pass diagonally to a subsequence on which the number of
operations and their dimensions are constant at each fixed level and
all required inner products converge. As usual, empty basis slots are
filled by zero vectors.

For a fixed selected subspace $M_n$ with orthonormal basis
$e_{n,1},\ldots,e_{n,m}$, define
$\mathcal M=\operatorname{span}\{[e_1],\ldots,[e_m]\}\subset\cH_q$.
The normalization gives the exact transfer identity
\begin{equation}
 [P_{M_n}k_n]=P_{\mathcal M}[k]
 =\sum_{j=1}^m\left(\lim_n\langle k_n,e_{n,j}\rangle\right)[e_j].
 \label{ext:projection-transfer}
\end{equation}
Indeed, put $a_{n,j}=\langle k_n,e_{n,j}\rangle$ and
$a_j=\langle[k],[e_j]\rangle$. Proposition~\ref{str:sequence-spaces}
gives $a_{n,j}\to a_j$, and orthonormality gives
\[
 \left\|P_{M_n}k_n-\sum_{j=1}^m a_je_{n,j}\right\|^2
 =\sum_{j=1}^m|a_{n,j}-a_j|^2\longrightarrow0.
\]
The classes $[e_j]$ are orthonormal, proving
\eqref{ext:projection-transfer} for every representative.

\begin{lemma}[Selected directions are decomposable]
\label{ext:selected-subspaces}
At each fixed level $J$, the selected subspace in the sequence space
satisfies $\mathcal M^{[J]}\subset\cD_q$. Equivalently, for every
$[k]\in\cG_q$ and every bounded sequence $u_n\in M_n^{[J]}$,
$\langle k_n,u_n\rangle\to0$.
\end{lemma}
\begin{proof}
Fix $J$. There are at most $J-1$ operations. The initial subspace is
zero. Suppose the assertion holds for the current $M_n$, and let
$V_n=\mathsf B_nT_n(V_n^{\rm top})$ be the next increment.
For a bounded $u_n\in V_n$, the gap
\eqref{ext:fixed-level-gap} supplies $c_n\in V_n^{\rm top}$ with
\[
 u_n=\mathsf B_nT_nc_n,\qquad
 \|c_n\|\le\frac{J}{a_qR^3}\|u_n\|.
\]
(The case $R=0$ has no increment.)
Let $[k]\in\cG_q$. The current basis has finitely many classes in
$\cD_q$, so \eqref{ext:projection-transfer} yields
$\|P_{M_n}k_n\|\to0$. Consequently
$\widetilde k_n=P_{M_n^\perp}k_n$ also represents $[k]$, and
\[
 \varepsilon_n=\max_{1\le s<q}
                 \|T_{\widetilde k_n,s}\|_{\op}\longrightarrow0.
\]
For a bounded order-$q$ tensor $v_n$,
$\|v_n\wtensor_s\widetilde k_n\|\le\|v_n\|\varepsilon_n$.
Summing the terms defining $\Lambda_n$ therefore gives
\[
 \|\Lambda_nk_n\|^2
 \le\varepsilon_n^2\sum_a(\|r_{n,a}\|^2+\|h_{n,a}\|^2)
                    \sum_{s=1}^{q-1}b_{q,q,s}
 \le a_qR^2\varepsilon_n^2.
\]
Since $\|T_nc_n\|\le R\|c_n\|$,
\[
 |\langle k_n,u_n\rangle|
 =|\langle\Lambda_nk_n,\Lambda_nT_nc_n\rangle|
 \le a_qR^2\varepsilon_n\|T_nc_n\|
 \le J\varepsilon_n\|u_n\|\longrightarrow0.
\]
A bounded vector in $M_n\oplus V_n$ has bounded orthogonal
projections on both summands. The induction hypothesis and this
estimate prove the assertion on the enlarged space. In particular,
its basis classes are orthogonal to every element of $\cG_q$, hence
belong to $\cD_q$.
\end{proof}

\begin{theorem}[Limits of the spectral projections]
\label{ext:homogeneous}
On the subsequence just constructed, put $\xi_a=[x_a]\in\cH_q$.
The classes of the actual kernel projections satisfy
\begin{equation}
 [x_a^{\mathrm{act},J}]\longrightarrow P_{\cD_q}\xi_a,
 \qquad
 [x_a^{\mathrm{res},J}]\longrightarrow P_{\cG_q}\xi_a
 \quad\text{in }\cH_q.
 \label{ext:projection-limits}
\end{equation}
Define
\[
 F^{\mathrm{act},\infty}_a=\mathbb U_q^{-1}P_{\cD_q}\xi_a,
 \qquad
 G_a=\mathbb U_q^{-1}P_{\cG_q}\xi_a.
\]
Then $G$ is Gaussian, independent of $F^{\mathrm{act},\infty}$, and
\begin{equation}
 C_{ab}=\mathbb E[G_aG_b]
 =\langle P_{\cG_q}\xi_a,P_{\cG_q}\xi_b\rangle_{\cH_q}.
 \label{ext:gaussian-covariance}
\end{equation}
For a sufficiently slowly increasing integer sequence $J(n)$,
\begin{equation}
 (F_n^{\mathrm{act},J(n)},F_n^{\mathrm{res},J(n)})
 \Longrightarrow(F^{\mathrm{act},\infty},G),
 \qquad F_n\Longrightarrow F^{\mathrm{act},\infty}+G.
 \label{ext:homogeneous-law}
\end{equation}
\end{theorem}
\begin{proof}
Put $\alpha_a^J=[x_a^{\mathrm{act},J}]$ and
$\rho_a^J=[x_a^{\mathrm{res},J}]$.
By \eqref{ext:projection-transfer}, these are the orthogonal
projections of $\xi_a$ onto $\mathcal M^{[J]}$ and its complement.
For $K\ge J$, nestedness gives
\[
 \sum_a\|\alpha_a^K-\alpha_a^J\|^2
 =\sum_a\|\alpha_a^K\|^2-\sum_a\|\alpha_a^J\|^2.
\]
The right side tends to zero as $J,K\to\infty$, since the sums
increase and are bounded by $R^2$. Thus
$\alpha^J\to\alpha^\infty\in\cD_q^d$ by
Lemma~\ref{ext:selected-subspaces}, and
$\rho^J\to\rho^\infty=\xi-\alpha^\infty$.

The fourth-cumulant identities
\cite[Lemma~5.3]{NourdinLectures}, with $Y=I_q(y/\sqrt{q!})$, give
\begin{equation}
 \binom qs^2\|y\otimes_sy\|^2\le\kappa_4(Y)
 \le3q\Var(\Gamma_{\rm MS}(Y,Y)),\qquad 1\le s<q.
 \label{ext:fourth-control}
\end{equation}
For the upper bound, the coefficient of
$\|y\wtensor_sy\|^2$ in the cumulant formula is $3q/s$ times
its coefficient in the variance formula, and $s\ge1$.
For the lower bound, retain the indicated nonnegative unsymmetrized
term in the other cumulant formula.
From \eqref{ext:pure-reduced-bounds} and
\eqref{str:singular-value-estimate},
\[
 \limsup_n\|T_{x_{n,a}^{\mathrm{res},J},s}\|_{\op}
 \le\left(\frac{3q d a_qR^4}{J\binom qs^2}\right)^{1/4}.
\]
Choose a representative $z_{n,a}$ of $\rho_a^\infty$.
The operator-norm Lipschitz bound gives, for every fixed $J$,
\[
 \limsup_n\|T_{z_{n,a},s}\|_{\op}
 \le\|\rho_a^\infty-\rho_a^J\|
   +\left(\frac{3q d a_qR^4}{J\binom qs^2}\right)^{1/4}.
\]
Let $J\to\infty$. All proper operator norms vanish, so
$\rho_a^\infty\in\cG_q$ by
Theorem~\ref{str:contraction-criterion}. Together with
$\alpha_a^\infty\in\cD_q$, this proves
\eqref{ext:projection-limits}.

Under $\mathbb U_q^{-1}$, $\cG_q$ is the linear Gaussian space over
$E_q$. A weighted-degree-$q$ Wick monomial with at least two factors
uses only $E_s$, $s<q$. Taking $L^2$ limits shows that
$F^{\mathrm{act},\infty}$ is measurable with respect to the field over
$\bigoplus_{s<q}E_s$, independent of the field over $E_q$.
Isometry gives \eqref{ext:gaussian-covariance}.

Let $Z_n^J=(F_n^{\mathrm{act},J},F_n^{\mathrm{res},J})$ and
$Z^J=(\mathbb U_q^{-1}\alpha_a^J,\mathbb U_q^{-1}\rho_a^J)_a$,
with the coordinates grouped in the same order.
For fixed $J$, Theorem~\ref{str:law-transfer} gives $Z_n^J\weak Z^J$,
whereas $Z^J\to Z^\infty=(F^{\mathrm{act},\infty},G)$ in $L^2$.
Write $d_{\rm BL}$ for the supremum over test functions bounded by
one and with Lipschitz constant at most one. Choose $N_J\uparrow\infty$
such that
\[
 d_{\rm BL}(\Law(Z_n^J),\Law(Z^J))\le J^{-1}\quad(n\ge N_J).
\]
For $N_J\le n<N_{J+1}$, set $J(n)=J$. Then
\[
 d_{\rm BL}(\Law(Z_n^{J(n)}),\Law(Z^\infty))
 \le J(n)^{-1}+\|Z^{J(n)}-Z^\infty\|_2\longrightarrow0.
\]
This proves the first convergence in \eqref{ext:homogeneous-law};
addition of the two coordinates proves the second.
\end{proof}

\begin{example}[Why the spectral restriction is needed]
\label{ext:selection-example}
Let $u=e_0^{\otimes2}$ and
$v_n=n^{-1/2}\sum_{j=1}^ne_j^{\otimes2}$ for orthonormal $e_j$.
Take the two order-two input kernels $u,v_n$. Identifying symmetric
order-two tensors with self-adjoint Hilbert--Schmidt operators, their
initial extraction operator satisfies
\[
 \mathsf B_nu=2u,\qquad \mathsf B_nv_n=2v_n/n.
\]
Indeed $x\wtensor_1y=(xy+yx)/2$, $uv_n=0$, and $v_n^3=v_n/n$.
When an operation is performed, simultaneous extraction therefore
selects $\operatorname{span}\{u,v_n\}$. However $[u]\in\cD_2$ whereas
$\|T_{v_n,1}\|_{\rm op}=n^{-1/2}\to0$ and $\|v_n\|=1$, so
$[v]\in\cG_2$. Spectral extraction excludes this vanishing-eigenvalue
direction at every fixed-level for sufficiently large $n$.
\end{example}

\begin{example}[No uniform rate for the limiting projections]
\label{ext:scope}
The convergence in \eqref{ext:projection-limits} has no uniform rate
over bounded inputs. To see this, fix $N$ and consider the
constant sequence
$x=N^{-1/2}\sum_{j=1}^Ne_j^{\otimes2}$, including the $e_j$ among the
represented sequences. Then $[x]\in\cD_2$, $\|x\|=1$, while
$\mathcal S(0,x)=2/N$. At any level $J\le N$, the spectral threshold
is $2/J$, no direction is selected, and the error in the active
projection is still $1$. Thus the rate $J^{-1/2}$ concerns Gaussian
replacement at a fixed-level, not convergence to the decomposable projection.
The diagonal choice $J(n)$ must likewise be made after the fixed
level limits: increasing it too quickly can select diffuse primitive
directions such as $v_n$ in Example~\ref{ext:selection-example}.
\end{example}

\section{Realization by multiple Wiener integrals}
\label{rec:section}

We realize finite-degree weighted Gaussian expansions by rescaled
kernels on $H=L^2(\mathbb R)$. The Gaussian limit of the distinct-index
sums used below is classical \cite{DynkinMandelbaum,AvramTaqqu}; the
construction also keeps exact inner products for finite expansions and
the contraction data of Theorem~\ref{str:contractions}.

Fix $Q<\infty$ and separable spaces $E_s$, $1\le s\le Q$.
By adjoining unused orthogonal coordinates to $E_1$ when necessary,
we may assume that $E$ is infinite-dimensional. This does not change
the prescribed expansions or their contractions.
Enumerate orthonormal bases of the direct sum as $(e_c)_{c\ge1}$,
write $s_c$ for the weight of $e_c$, and put $G_c=W(e_c)$.
Choose unit functions $\phi_{c,j,\ell}\in C_c^\infty(\mathbb R)$,
$j\ge1$, $1\le\ell\le s_c$, with mutually separated supports: the
distance between any two supports is at least $2$.
Translated copies of one function supported in $(0,1)$ suffice.
Define the rescaled kernels
\begin{equation}
 h_{c,j}=\sqrt{s_c!}\,\operatorname{Sym}
       (\phi_{c,j,1}\otimes\cdots\otimes\phi_{c,j,s_c}),
 \qquad u_{N,c}=N^{-1/2}\sum_{j=1}^Nh_{c,j}.
 \label{rec:blocks}
\end{equation}
Thus $\|h_{c,j}\|=1$ and
$I_{s_c}(h_{c,j}/\sqrt{s_c!})
=\prod_{\ell=1}^{s_c}W_0(\phi_{c,j,\ell})$, where $W_0$ is the
isonormal process over $H$ used for this realization.

\begin{lemma}[Orthogonal copies of primitive coordinates]
\label{rec:replicas}
The kernels in \eqref{rec:blocks} satisfy $\|u_{N,c}\|=1$ and
\begin{equation}
 \|u_{N,c}\otimes_r u_{N,c}\|
   \le N^{-1/2},\qquad 1\le r<s_c.
 \label{rec:copy-contraction}
\end{equation}
Every positive-order contraction between $u_{N,c}$ and $u_{N,d}$
vanishes when $c\ne d$.
Consequently any finite vector
$(I_{s_c}(u_{N,c}/\sqrt{s_c!}))_c$ converges to independent standard
Gaussians with the weights $s_c$.
\end{lemma}

\begin{proof}
Different pairs $(c,j)$ use different functions $\phi_{c,j,\ell}$, and
their supports are pairwise separated. For distinct orthonormal
functions,
\[
 \bigl\|\operatorname{Sym}(\phi_1\otimes\cdots\otimes\phi_s)\bigr\|^2
 =\frac1{(s!)^2}\sum_{\pi,\pi'\in\mathfrak S_s}
   \prod_{\ell=1}^s\langle\phi_{\pi(\ell)},\phi_{\pi'(\ell)}\rangle
 =\frac1{s!},
\]
since only $\pi=\pi'$ contributes. Hence $\|h_{c,j}\|=1$. Kernels
with different $j$ have disjoint supports, so they are orthogonal and
\[
 \|u_{N,c}\|^2=\frac1N\sum_{j=1}^N\|h_{c,j}\|^2=1.
\]
A contraction integrates a product of the two kernels over the
contracted variables. If $j\ne j'$, the integrand vanishes identically,
so $h_{c,j}\otimes_rh_{c,j'}=0$ for $r\ge1$. Therefore
\begin{equation}
 u_{N,c}\otimes_ru_{N,c}=\frac1N\sum_{j=1}^Nh_{c,j}\otimes_rh_{c,j}.
 \label{rec:copy-expansion}
\end{equation}
For $r<s_c$, the $j$-th summand is supported in the product of the
supports of block $j$. The summands are thus orthogonal. Each has norm
at most $\|h_{c,j}\|^2=1$. Hence
\[
 \|u_{N,c}\otimes_ru_{N,c}\|^2
 =\frac1{N^2}\sum_{j=1}^N\|h_{c,j}\otimes_rh_{c,j}\|^2
 \le\frac1N,
\]
by \eqref{pre:contraction-bound}. This proves
\eqref{rec:copy-contraction}. For $c\ne d$ all functions involved are
different, and the same support argument gives
$u_{N,c}\otimes_ru_{N,d}=0$ for $r\ge1$.

For the last assertion put $Z_{c,j}=I_{s_c}(h_{c,j}/\sqrt{s_c!})$.
Then
\begin{equation}
 Z_{c,j}=\prod_{\ell=1}^{s_c}W_0(\phi_{c,j,\ell}),
 \qquad
 I_{s_c}\!\left(\frac{u_{N,c}}{\sqrt{s_c!}}\right)
 =\frac1{\sqrt N}\sum_{j=1}^NZ_{c,j}.
 \label{rec:copy-sum}
\end{equation}
The Gaussian variables $W_0(\phi_{c,j,\ell})$ are independent and
standard. Thus the $Z_{c,j}$ are independent, centered, of variance
one, and have moments of every order. The multivariate central limit
theorem gives the joint convergence of finitely many of the sums in
\eqref{rec:copy-sum} to independent standard Gaussians.
\end{proof}

For a finitely supported multi-index $m=(m_c)_c$, set
\[
 H_m(G)=\prod_cH_{m_c}(G_c),\qquad
 d(m)=\sum_cs_cm_c,\qquad m!=\prod_cm_c!,\qquad
 D_{N,m}=\prod_c(N)_{m_c},
\]
where $(N)_k=N(N-1)\cdots(N-k+1)$ and $(N)_0=1$.
For $N\ge Q$ and $d(m)=q\le Q$, define
\begin{equation}
 x_{N,m}=D_{N,m}^{-1/2}
 \sum_{\substack{1\le j_{c,1},\ldots,j_{c,m_c}\le N\\
                 j_{c,1},\ldots,j_{c,m_c}\ \mathrm{distinct\ for\ each}\ c}}
 \mathop{\diamond}_{c}\mathop{\diamond}_{k=1}^{m_c}h_{c,j_{c,k}}.
 \label{rec:distinct-kernel}
\end{equation}
The empty product is $1$. In particular $x_{N,0}=1$.

\begin{theorem}[Exact isometry for distinct-index Hermite kernels]
\label{rec:isometry}
For $N\ge Q$, the assignment
$\iota_{N,q}H_m=x_{N,m}$, $d(m)=q$, extends to an isometry
\[
 \iota_{N,q}:\cP_q(E)\longrightarrow H^{\odot q}.
\]
For every finite Hermite expansion $P\in\cP_q(E)$, the kernel
$\iota_{N,q}P$ is smooth, compactly supported, and vanishes whenever
$|t_i-t_j|<1$ for some $i\ne j$.
\end{theorem}

\begin{proof}
Associativity and \eqref{pre:contraction-bound} give
\begin{equation}
 \mathop{\diamond}_{i=1}^k\Bigl(\sqrt{s_i!}\,
 \operatorname{Sym}(\psi_{i,1}\otimes\cdots\otimes\psi_{i,s_i})\Bigr)
 =\sqrt{q!}\,\operatorname{Sym}\Bigl(\bigotimes_{i,\ell}\psi_{i,\ell}\Bigr),
 \qquad q=\sum_is_i.
 \label{rec:diamond-blocks}
\end{equation}
In a summand of \eqref{rec:distinct-kernel}, the blocks $(c,j_{c,k})$
are pairwise distinct. The $q$ functions in \eqref{rec:diamond-blocks}
are therefore distinct and orthonormal. The norm computation in the
proof of Lemma~\ref{rec:replicas} gives
\[
 \|T\|^2=q!\cdot\frac1{q!}=1
\]
for every such summand $T$.

A summand depends only on the set $S$ of blocks it uses, since
$\diamond$ is commutative. Two different sets use two different
families of functions. The corresponding symmetrized tensors are
orthogonal. For each $c$ there are $(N)_{m_c}$ ordered choices, hence
$D_{N,m}$ ordered choices in total. Each set $S$ arises from exactly
$m!$ of them. Writing $T_S$ for its summand,
\begin{equation}
 x_{N,m}=\frac{m!}{\sqrt{D_{N,m}}}\sum_ST_S,
 \qquad
 \|x_{N,m}\|^2=\frac{(m!)^2}{D_{N,m}}\cdot\frac{D_{N,m}}{m!}=m!.
 \label{rec:distinct-norm}
\end{equation}

Let $m\ne m'$ with $d(m)=d(m')$. Some $c$ satisfies $m_c\ne m'_c$.
Then every $T_S$ in $x_{N,m}$ and every $T_{S'}$ in $x_{N,m'}$ use
different numbers of blocks of type $c$. Their function families
differ, and $\langle x_{N,m},x_{N,m'}\rangle=0$.

The Hermite products $H_m(G)$ with $d(m)=q$ form an orthogonal basis of
$\cP_q(E)$, and
\begin{equation}
 \E[H_m(G)H_{m'}(G)]=\1_{\{m=m'\}}\,m!
 \label{rec:hermite-gram}
\end{equation}
\cite[Chapter~2]{Janson}. By \eqref{rec:distinct-norm} and the
orthogonality just proved, $\iota_{N,q}$ maps this orthogonal basis to
an orthogonal family with the same norms. It is therefore isometric on
finite expansions and extends by continuity.

For a finite expansion, $\iota_{N,q}P$ is a finite linear combination
of tensors $T_S$. Each $T_S$ is a finite sum of products
$\prod_{i=1}^q\psi_i(t_{\pi(i)})$ of distinct functions from the
family $\phi_{c,j,\ell}$. These functions are smooth with compact
support. Two distinct ones have supports at distance at least $2$.
If $|t_i-t_j|<1$ with $i\ne j$, the variables $t_i,t_j$ cannot lie in
the supports of two distinct functions, and every product vanishes.
\end{proof}

Let $\Pi_M$ denote conditional expectation onto
$\sigma(G_1,\ldots,G_M)$. It preserves each weighted degree and
$\Pi_MF$ is a finite Hermite expansion whenever
$F\in\bigoplus_{q=0}^Q\cP_q(E)$.

\begin{theorem}[Simultaneous realization of finite-degree expansions]
\label{rec:realization}
Let $F=(F_1,\ldots,F_d)$ with
$F_a=\sum_{q=0}^QF_{a,q}$ and $F_{a,q}\in\cP_q(E)$.
For $N\ge Q$, set
\begin{equation}
 Y_{N,a}=\sum_{q=0}^Q
 I_q\left(\frac{\iota_{N,q}(\Pi_NF_{a,q})}{\sqrt{q!}}\right).
 \label{rec:realizing-vector}
\end{equation}
Then $Y_N\Rightarrow F$, jointly with convergence of all mixed
moments, and
\begin{equation}
 \mathbb E[Y_{N,a}Y_{N,b}]
       =\mathbb E[(\Pi_NF_a)(\Pi_NF_b)].
 \label{rec:exact-gram}
\end{equation}
The kernels in \eqref{rec:realizing-vector} have the smoothness and
support properties of Theorem~\ref{rec:isometry}.
If $F_a\in\cP_q(E)$, its realization belongs to the single chaos of
order $q$. If all $F_a$ are finite Hermite expansions, their entire
Gram matrix is exact once $N$ includes their coordinates.
\end{theorem}

\begin{proof}
\emph{Finite expansions.}
Let $Z_{c,j}$ be as in \eqref{rec:copy-sum}. Let $S$ be a set of
distinct blocks with $m_c$ blocks of type $c$. By
\eqref{rec:diamond-blocks}, $T_S=\sqrt{q!}\operatorname{Sym}(\otimes\psi)$
with distinct orthonormal functions $\psi$. The multiple integral of a
symmetrized tensor product of distinct orthonormal functions is the
product of their Wiener integrals. Hence
$I_q(T_S/\sqrt{q!})=\prod_{(c,j)\in S}Z_{c,j}$, and
\begin{equation}
 I_{d(m)}\!\left(\frac{x_{N,m}}{\sqrt{d(m)!}}\right)
 =\prod_cU_{N,c,m_c},
 \qquad
 U_{N,c,k}=\frac1{\sqrt{(N)_k}}
   \sum_{\substack{j_1,\ldots,j_k\le N\\ \text{distinct}}}
   Z_{c,j_1}\cdots Z_{c,j_k}.
 \label{rec:symmetric-sums}
\end{equation}
Fix $c$ and $k$, and write $Z_j=Z_{c,j}$ and
$p_\ell=\sum_{j\le N}Z_j^\ell$. Möbius inversion on the lattice of set
partitions gives the Newton identity
\begin{equation}
 \sum_{\substack{j_1,\ldots,j_k\ \text{distinct}}}
 Z_{j_1}\cdots Z_{j_k}
 =\sum_{\pi}\ \prod_{B\in\pi}(-1)^{|B|-1}(|B|-1)!\,p_{|B|},
 \label{rec:newton}
\end{equation}
where $\pi$ ranges over the partitions of $\{1,\ldots,k\}$.
To verify this identity, fix an index tuple with equality partition
$\rho$. Its coefficient on the right is
\[
 \sum_{\pi\le\rho}\prod_{B\in\pi}(-1)^{|B|-1}(|B|-1)!
 =\prod_{A\in\rho}\sum_{\tau\in\mathfrak S_{|A|}}
                                  (-1)^{|A|-\#\text{cycles}(\tau)}.
\]
Each inner sum is the sum of permutation signs: it is $1$ for a
singleton and $0$ for $|A|\ge2$, by pairing a permutation with its
composition with a transposition. Thus only distinct index tuples
have coefficient one. Divide \eqref{rec:newton} by $N^{k/2}$. A singleton gives
$S_N=N^{-1/2}p_1$. A pair gives $N^{-1}p_2\to1$ almost surely, by the
law of large numbers. A block of size $\ell\ge3$ gives
\[
 \E\bigl|N^{-\ell/2}p_\ell\bigr|
 \le N^{1-\ell/2}\,\E|Z_1|^\ell\longrightarrow0.
\]
The singleton terms are tight by the central limit theorem, the pair
terms converge in probability to $1$, and every larger-block term
converges in probability to zero. Each partition has finitely many
blocks, so any product containing a larger block tends to zero in
probability. Only singletons and pairs remain. A pair has
coefficient $-1$, and there are $k!/(i!\,2^i(k-2i)!)$ partitions with
$i$ pairs. Since $(N)_k/N^k\to1$,
\begin{equation}
 U_{N,c,k}
 =\sum_{i\le k/2}\frac{(-1)^ik!}{i!\,2^i(k-2i)!}\,S_{N,c}^{\,k-2i}
   +o_{\mathbb P}(1)
 =H_k(S_{N,c})+o_{\mathbb P}(1),
 \label{rec:hermite-limit}
\end{equation}
where $S_{N,c}=N^{-1/2}\sum_{j\le N}Z_{c,j}$. By
Lemma~\ref{rec:replicas}, finitely many $S_{N,c}$ converge jointly to
independent standard Gaussians $G_c$. The continuous mapping theorem
and Slutsky's lemma, applied to \eqref{rec:symmetric-sums} and
\eqref{rec:hermite-limit}, give the joint convergence
\[
 \Bigl(I_{d(m)}\bigl(x_{N,m}/\sqrt{d(m)!}\bigr)\Bigr)_{m\in\mathcal M}
 \Longrightarrow\bigl(H_m(G)\bigr)_{m\in\mathcal M}
\]
for every finite set $\mathcal M$ of multi-indices. A finite
expansion uses finitely many $m$, and $\Pi_NF=F$ once $N$ exceeds the
indices of its coordinates. Linearity gives $Y_N\Rightarrow F$.

\emph{Approximation.}
Write $Y_N(P)$ for the vector \eqref{rec:realizing-vector} built from
$P$. For $N\ge M$ we have $\Pi_N\Pi_M=\Pi_M$. The isometry of
Theorem~\ref{rec:isometry} and orthogonality of the chaos orders give
\begin{equation}
 \E|Y_N(F)-Y_N(\Pi_MF)|^2
 =\|\Pi_N(F-\Pi_MF)\|_2^2\le\|F-\Pi_MF\|_2^2.
 \label{rec:approx-error}
\end{equation}
Let $g$ be bounded and $1$-Lipschitz. By \eqref{rec:approx-error},
\[
 |\E g(Y_N(F))-\E g(F)|
 \le2\|F-\Pi_MF\|_2
 +|\E g(Y_N(\Pi_MF))-\E g(\Pi_MF)|.
\]
The last term tends to zero as $N\to\infty$, by the finite-expansion case. The martingale
convergence theorem gives $\Pi_MF\to F$ in $L^2$. Letting first
$N\to\infty$ and then $M\to\infty$ proves $Y_N\Rightarrow F$.

\emph{Moments and Gram matrices.}
Every $Y_{N,a}$ lies in the chaoses of orders at most $Q$, and
$\E Y_{N,a}^2=\|\Pi_NF_a\|_2^2\le\|F_a\|_2^2$. By
Lemma~\ref{pre:hypercontractivity}, every fixed mixed moment is
uniformly integrable, hence convergent. Chaoses of different orders
are orthogonal, and so are the spaces $\cP_q(E)$ for different $q$.
Theorem~\ref{rec:isometry} therefore gives
\[
 \E[Y_{N,a}Y_{N,b}]
 =\sum_{q=0}^Q\langle\iota_{N,q}\Pi_NF_{a,q},\iota_{N,q}\Pi_NF_{b,q}\rangle
 =\sum_{q=0}^Q\E[\Pi_NF_{a,q}\,\Pi_NF_{b,q}]
 =\E[(\Pi_NF_a)(\Pi_NF_b)].
\]
This is \eqref{rec:exact-gram}. If $F_a\in\cP_q(E)$, only the term of
order $q$ appears in \eqref{rec:realizing-vector}. The support
properties follow from Theorem~\ref{rec:isometry}.
\end{proof}

We next verify contraction compatibility. For $d(m)=q$, let
\[
 y_{N,m}=\mathop{\diamond}_{c}u_{N,c}^{\diamond m_c},
 \qquad \widetilde\iota_{N,q}H_m=y_{N,m}.
\]
The tilde here distinguishes the unrestricted copy sum from the
distinct-index sum; both are rescaled kernels.

\begin{lemma}[Error from repeated copy indices]
\label{rec:comparison}
Write $\eta_q=q(q-1)/2$ and
$c_q^{\mathrm{rep}}=\eta_q+\sqrt{q!\eta_q}$.
For $N\ge Q$ and $P\in\cP_q(E)$,
\begin{equation}
 \|\widetilde\iota_{N,q}P-\iota_{N,q}P\|
      \le c_q^{\mathrm{rep}}N^{-1/2}\|P\|_2.
 \label{rec:kernel-comparison}
\end{equation}
In particular the unrestricted assignment extends continuously to
all of $\cP_q(E)$.
\end{lemma}

\begin{proof}
Fix $m$ with $d(m)=q$, put $M=\sum_cm_c$, and
$\eta(m)=\sum_cm_c(m_c-1)/2\le\eta_q$. Expanding the diamond powers,
\[
 y_{N,m}=N^{-M/2}\sum_{(j_{c,k})}\ \mathop{\diamond}_{c}
 \mathop{\diamond}_{k=1}^{m_c}h_{c,j_{c,k}},
\]
where the sum is now unrestricted. Split it into choices with distinct
indices for each $c$ and choices with a repetition. By
\eqref{rec:distinct-kernel},
\begin{equation}
 y_{N,m}=\alpha_{N,m}x_{N,m}+e_{N,m},
 \qquad
 \alpha_{N,m}=\sqrt{D_{N,m}/N^M},
 \label{rec:split}
\end{equation}
and $e_{N,m}$ collects the repeated choices. A uniformly random
choice has a repetition with probability
$1-D_{N,m}/N^M\le\eta(m)/N$, by the union bound over pairs. Since
$\alpha\in[0,1]$,
\[
 0\le1-\alpha_{N,m}\le1-\alpha_{N,m}^2\le\frac{\eta(m)}N.
\]
Group the repeated choices by their multisets of blocks. Different
multisets use different families of functions, with multiplicities,
and give orthogonal tensors. A multiset arises from at most $m!$
choices. The tensor of a multiset is a diamond product of unit
kernels of total order $q$, so its squared norm is at most $q!$ by
repeated application of \eqref{pre:contraction-bound}. Hence
\[
 \|e_{N,m}\|^2
 \le N^{-M}\,m!\,q!\,\bigl(N^M-D_{N,m}\bigr)
 \le q!\,m!\,\frac{\eta(m)}N.
\]
With $\|x_{N,m}\|=\sqrt{m!}=\|H_m\|_2$ and $\eta(m)/N\le\eta_qN^{-1/2}$,
\eqref{rec:split} gives
\begin{equation}
 \|y_{N,m}-x_{N,m}\|
 \le\Bigl(\eta_q+\sqrt{q!\,\eta_q}\Bigr)N^{-1/2}\|H_m\|_2
 =c_q^{\mathrm{rep}}N^{-1/2}\|H_m\|_2.
 \label{rec:monomial-error}
\end{equation}
Now let $m\ne m'$ with $d(m)=d(m')=q$. Some $c$ has $m_c\ne m'_c$.
Every tensor in $y_{N,m}-x_{N,m}$ contains exactly $m_c$ kernels of
type $c$, counted with multiplicity, and similarly for $m'$. These
differences are therefore orthogonal. For $P=\sum_ma_mH_m$ finite,
\[
 \|\widetilde\iota_{N,q}P-\iota_{N,q}P\|^2
 =\sum_ma_m^2\|y_{N,m}-x_{N,m}\|^2
 \le\bigl(c_q^{\mathrm{rep}}\bigr)^2N^{-1}\sum_ma_m^2\,m!
 =\bigl(c_q^{\mathrm{rep}}\bigr)^2N^{-1}\|P\|_2^2.
\]
This proves \eqref{rec:kernel-comparison} for finite expansions.
Since $\iota_{N,q}$ is isometric, $\widetilde\iota_{N,q}$ is bounded
on finite expansions and extends by continuity.
\end{proof}

For a diagram $\mathfrak d$ and finite Hermite inputs $P_v$ of weighted degrees
$q_v$, expand
\[\mathcal C_{\mathfrak d}(\widetilde\iota_{N,q_1}P_1,\ldots,
\widetilde\iota_{N,q_k}P_k)\]
into primitive factors $u_{N,c}$.
Retain the terms in which each contracted primitive factor is paired
in all its arguments with one factor corresponding to the same
Gaussian coordinate. Denote the resulting tensor by
$B_{\mathfrak d,N}(P_1,\ldots,P_k)$.
The coefficients of the diamond products and symmetrizations are
retained without modification. A pair of entire primitive factors contributes
$\langle u_{N,c},u_{N,d}\rangle=\mathbf1_{\{c=d\}}$; unpaired factors
remain in the tensor output.

\begin{proposition}[Preservation of contraction diagrams]
\label{rec:contractions}
For fixed finite Hermite inputs and a fixed diagram,
\begin{equation}
 \left\|\mathcal C_{\mathfrak d}(\iota_{N,q_1}P_1,\ldots,\iota_{N,q_k}P_k)
       -B_{\mathfrak d,N}(P_1,\ldots,P_k)\right\|
       \le c_{\mathfrak d,P_1,\ldots,P_k}N^{-1/2}.
 \label{rec:diagram-error}
\end{equation}
For scalar outputs, $B_{\mathfrak d,N}$ is independent of $N$
and equals the contraction prescribed by
Theorem~\ref{str:contractions}. For general
$F_v\in\cP_{q_v}(E)$, the realizing kernels
$\iota_{N,q_v}\Pi_NF_v$ have the same limiting scalar diagram values.
For open diagrams, the construction also preserves the limiting
scalar products of outputs, and hence their norms.
\end{proposition}

\begin{proof}
We use the multilinear telescoping identity
\begin{equation}
 \mathcal C_{\mathfrak d}(a_1,\ldots,a_k)-\mathcal C_{\mathfrak d}(b_1,\ldots,b_k)
 =\sum_{v=1}^k\mathcal C_{\mathfrak d}(b_1,\ldots,b_{v-1},a_v-b_v,a_{v+1},\ldots,a_k).
 \label{rec:telescoping}
\end{equation}
By Lemma~\ref{str:diagram-bound}, the norm of each term is at most
$\|a_v-b_v\|\prod_{u\ne v}\max\{\|a_u\|,\|b_u\|\}$.
Take $a_v=\iota_{N,q_v}P_v$ and $b_v=\widetilde\iota_{N,q_v}P_v$. Then
$\|a_v\|=\|P_v\|_2$, and by \eqref{rec:kernel-comparison}
\[
 \|b_v\|\le(1+c_{q_v}^{\mathrm{rep}}N^{-1/2})\|P_v\|_2,\qquad
 \|a_v-b_v\|\le c_{q_v}^{\mathrm{rep}}N^{-1/2}\|P_v\|_2.
\]
Consequently the difference between these two diagrams is at most
\[
 N^{-1/2}\left(\sum_{v=1}^k c_{q_v}^{\mathrm{rep}}
              \prod_{u\ne v}(1+c_{q_u}^{\mathrm{rep}})\right)
                \prod_{u=1}^k\|P_u\|_2,
\]
uniformly for $N\ge Q$. It remains to treat unrestricted copy sums.

Each $\widetilde\iota_{N,q_v}P_v$ is a fixed finite combination of
diamond products of the unit kernels $u_{N,c}$. Expand these diamond
products and their symmetrizations. We obtain a finite sum, with
coefficients independent of $N$. Each term assigns the slots of every
factor $u_{N,c}$ to slots of its vertex, and the diagram then
connects slots of different vertices. We sort the terms into three
classes.

(a) Some edge joins factors $u_{N,c}$ and $u_{N,d}$ with $c\ne d$.
The term vanishes by Lemma~\ref{rec:replicas}.

(b) Every contracted factor has all its slots joined to slots of a
single factor of the same type $c$. Each such pair contributes
$\langle u_{N,c},u_{N,c}\rangle=1$. These are the retained terms.
Their sum is $B_{\mathfrak d,N}$.

(c) All other terms. In such a term some factor $u_{N,c}$ has
$r\ge1$ slots joined to a factor $u_{N,c}$ at another vertex, with
$r<s_c$. Otherwise it would be in class (a) or (b). Contract these $r$
edges first. By \eqref{rec:copy-contraction}, the resulting tensor
has norm at most $N^{-1/2}$. The remaining factors have norm one. The
second form of Lemma~\ref{str:diagram-bound} bounds the term by
$N^{-1/2}$.

There are finitely many terms, which proves \eqref{rec:diagram-error}.
For a closed diagram every factor is contracted. Class (b) then
consists of complete pairings of whole factors of equal type, each
of value one. This value does not depend on $N$. It coincides with
the rule of Theorem~\ref{str:contractions}, where a pairing of
coordinates contributes $\langle e_c,e_d\rangle_E=\1_{\{c=d\}}$.

For general inputs $F_v\in\cP_{q_v}(E)$, apply \eqref{rec:telescoping}
with $a_v=\iota_{N,q_v}\Pi_NF_v$ and $b_v=\iota_{N,q_v}\Pi_MF_v$, where
$N\ge M$. The isometry and $\|\Pi_NF_v-\Pi_MF_v\|_2\le\|F_v-\Pi_MF_v\|_2$
give
\[
 \bigl\|\mathcal C_{\mathfrak d}(\iota_{N,q_v}\Pi_NF_v:v)
       -\mathcal C_{\mathfrak d}(\iota_{N,q_v}\Pi_MF_v:v)\bigr\|
 \le\sum_v\|F_v-\Pi_MF_v\|_2\prod_{u\ne v}\|F_u\|_2 .
\]
For a scalar output and fixed $M$, the second diagram converges to
the retained value for $\Pi_MF_v$, by the first part. The same estimate, applied to the
retained values, shows that they converge as $M\to\infty$. Letting
first $N\to\infty$ and then $M\to\infty$ gives the asserted scalar
limits. Finally, the scalar product of two open outputs is a closed
diagram on two disjoint copies of their input lists. The closed case
therefore gives the asserted limiting scalar products and norms.
\end{proof}

\begin{proof}[Proof of Theorem~\ref{rec:closure}]
By Lemma~\ref{pre:tightness}, a weakly convergent sequence at fixed
order has bounded kernel norms. Theorem~\ref{str:law-transfer} therefore
gives necessity, and Theorem~\ref{rec:realization} gives sufficiency
with the stated smoothness and support properties. For the last
assertion, choose integers
$k(n)\to\infty$ such that
$k(n)\sum_{c\le k(n)}s_c\le\dim H_n$ and replace the corresponding
functions $\phi_{c,j,\ell}$ by that many orthonormal vectors in $H_n$.
Apply the construction with $N=k(n)$ and the first $k(n)$ Gaussian
coordinates. Every finite expansion is eventually included and all
inner-product and contraction identities are unchanged.
\end{proof}

\section{Limits with exponential moments}
\label{light:section}

An absolute exponential moment forces a weighted Gaussian
representation to have ordinary degree at most two. The weights then
determine the possible linear and quadratic terms. This yields a
classification by parity and an exact criterion for quadratic Gaussian
targets.

\subsection{Classification by parity}

\begin{definition}[Laws attainable at a fixed order]
\label{lt:def-laws}
For $q\geq1$, let $\mathfrak L_q$ be the set of laws of limits in distribution of sequences
\[
 F_n=I_q\bigl(x_n/\sqrt{q!}\bigr).
\]
The underlying isonormal spaces may vary with $n$; the kernel norms
are automatically bounded by Lemma~\ref{pre:tightness}. Write $\mathfrak
L_q^{\exp}$ for the laws in $\mathfrak L_q$ satisfying $\mathbb
E\exp(c|F|)<\infty$ for some $c>0$.
\end{definition}

\begin{theorem}[Parity classification under an exponential moment]
\label{light:classification}
The following lists describe $\mathfrak L_q^{\exp}$ exactly:
\begin{align*}
 q=1:&\qquad \sigma G_0;\\
 q\geq2\text{ even}:&\qquad
 \sigma G_0+\sum_{j\geq1}\lambda_j(G_j^2-1),
 \qquad \sum_j\lambda_j^2<\infty;\\
 q\geq3\text{ odd}:&\qquad
 \sigma G_0+\sum_{j\geq1}b_jG_j\widetilde G_j,
 \qquad \sum_jb_j^2<\infty.
\end{align*}
Here $\sigma\geq0$, and all displayed Gaussian variables are standard
and mutually independent. Both series converge in $L^2$. In particular,
all even orders $q\geq2$ have the same class of limits with an
exponential moment, and all odd orders $q\geq3$ have the same such
class.
\end{theorem}

The central quadratic class in the second line is the known closure of
the second Wiener chaos; see \cite[Theorem~3.1]{NourdinPolySecond}. The
theorem identifies this same class in every fixed even order and
the bilinear subclass selected by odd weighted degree.

\begin{proof}
Let $P\in\cP_q(E)$ represent an attainable law with
$\E e^{c|P|}<\infty$. In an orthonormal coordinate system adapted to
$E=\bigoplus_s E_s$, let $P_N$ be conditional expectation onto the
first $N$ coordinates. It is a finite Hermite polynomial of weighted
degree $q$, and conditional Jensen gives
\[
 \E e^{c|P_N|}\le\E e^{c|P|}<\infty.
\]
Suppose $P_N$ has ordinary degree $d>2$, with leading homogeneous
part $p_d\ne0$. Choose a spherical set $U$ of positive surface
measure and $a>0$ such that $|p_d(\theta)|\ge a$ on $U$.
The lower-degree terms are bounded by $Cr^{d-1}$ on $rU$, $r\ge1$.
For $r\ge r_0=\max(1,2C/a)$,
$|P_N(r\theta)|\ge ar^d/2$ on $U$, whence
\[
 \E e^{c|P_N|}\ge\frac{|U|}{(2\pi)^{N/2}}
 \int_{r_0}^{\infty}r^{N-1}
                \exp\left(\frac{ca}{2}r^d-\frac{r^2}{2}\right)\,dr
 =\infty.
\]
For $N=1$, the same formula uses either half-line and counting measure
on the sphere. This contradiction proves $\deg P_N\le2$.
Since $P_N\to P$ in $L^2$, projection onto each ordinary chaos of
order greater than two gives zero. Also $\E P=0$ because $q>0$.

Write $P=L+Q$ as its ordinary linear and quadratic parts.
The weight constraint puts $L=W(\ell)$ with $\ell\in E_q$, and every
term of $Q$ pairs $E_r$ with $E_s$ for $r+s=q$, $r,s\ge1$.
In particular $Q$ uses only $\bigoplus_{s<q}E_s$, so $L$ and $Q$
are independent. For $q=1$, $Q=0$ and $L$ has law $\sigma G_0$.

For even $q$, write $Q=I_2^W(K)$, where $K$ defines a self-adjoint
Hilbert--Schmidt operator on $\bigoplus_{s<q}E_s$.
The spectral theorem and the Wiener isometry give
\[
 Q=\sum_j\lambda_j(G_j^2-1)\quad\hbox{in }L^2,
 \qquad \|Q\|_2^2=2\sum_j\lambda_j^2.
\]
For odd $q$, there is no quadratic term within one $E_r$.
In orthonormal coordinates,
\[
 Q=\sum_{1\le r<q/2}\sum_{i,j}
           a_{ij}^{(r)}G_{r,i}G_{q-r,j},\qquad
 \|Q\|_2^2=\sum_{r,i,j}|a_{ij}^{(r)}|^2.
\]
Each coefficient matrix is Hilbert--Schmidt. Its singular value
decomposition gives independent products
$\sum_j b_j^{(r)}G_{r,j}'G_{q-r,j}'$ with
$\sum_j(b_j^{(r)})^2=\sum_{i,j}|a_{ij}^{(r)}|^2$.
The pairs of weight spaces for different $r$ are disjoint, so all
coordinates can be taken independent. Enumerating the products gives
the asserted odd-order series.

Conversely, assign weight $q/2$ to the quadratic coordinates in the
even case and weights $1,q-1$ to the two factors in the odd case.
Assign weight $q$ to the independent $G_0$. Square summability puts
the displayed expansions in $\cP_q(E)$, and
Theorem~\ref{rec:realization} realizes them at order $q$.

It remains to check their exponential moments. Gaussian integration
gives
\begin{align*}
 \log\E e^{z\lambda(G^2-1)}
   &=-z\lambda-\tfrac12\log(1-2z\lambda),\\
 \log\E e^{zbG\widetilde G}
   &=-\tfrac12\log(1-b^2z^2).
\end{align*}
The coefficient sequences are bounded. Choose $c>0$ such that
$2c\sup_j|\lambda_j|\le1/2$ or $c\sup_j|b_j|\le1/2$, respectively.
For $|z|\le c$, the power series for $-\log(1-u)$ give
\[
 0\le-z\lambda-\tfrac12\log(1-2z\lambda)\le2z^2\lambda^2,
 \qquad
 0\le-\tfrac12\log(1-b^2z^2)\le z^2b^2.
\]
Thus, for the partial sums $S_M$ including $\sigma G_0$,
\[
 \sup_M\E e^{zS_M}\le
 \begin{cases}
 \exp\{z^2(\sigma^2/2+2\sum_j\lambda_j^2)\},&\text{even case},\\
 \exp\{z^2(\sigma^2/2+\sum_jb_j^2)\},&\text{odd case}.
 \end{cases}
\]
Along an almost surely convergent subsequence of $S_M\to P$ in
$L^2$, Fatou's lemma applies for $z=c$ and $z=-c$.
Since $e^{c|P|}\le e^{cP}+e^{-cP}$, the required absolute exponential
moment is finite. The Gaussian case is immediate.
\end{proof}

\subsection{Quadratic Gaussian targets and a Hankel matrix criterion}

Krein~\cite[Remark~4.10(2)]{Krein2019} discusses whether a
noncentral quadratic variable $c(G^2-1)+dG$, with $cd\ne0$, can
occur as a limit in a fixed Wiener chaos. The following criterion
excludes this case and characterizes all finite-dimensional quadratic
Gaussian targets.

\begin{theorem}[Quadratic targets and a finite cumulant criterion]
\label{light:quadratic}\label{light:hankel}
Let $X\sim\mathcal N(0,I_N)$, let $\mathsf A$ be a real symmetric
$N\times N$ matrix, and let $b\in\mathbb R^N$. For
\[
 F=X^\top\mathsf A X-\operatorname{tr}\mathsf A+b^\top X,
\]
the following statements are equivalent:
\begin{enumerate}
\item $\operatorname{Law}(F)\in\mathfrak L_q$ for at least one fixed order $q\geq1$;
\item $\operatorname{Law}(F)\in\mathfrak L_q$ for every even order $q\geq2$;
\item $\mathsf A b=0$.
\end{enumerate}
If $\mathsf A\neq0$, let $t$ be the number of distinct nonzero squared
eigenvalues of $\mathsf A$. Define, for $m\geq2$,
\[
 s_m=\frac{\kappa_{2m}(F)}{2^{2m-1}(2m-1)!},
 \qquad
 \mathsf H_t=(s_{i+j+2})_{0\leq i,j\leq t}.
\]
The preceding statements are also equivalent to $\mathsf H_t\succeq0$.
Thus cumulants through order $4t+4$ give a complete test within this
class of targets.
\end{theorem}

\begin{proof}
If $\mathsf A=0$, the target is Gaussian and all three assertions
hold. Suppose $\mathsf A\ne0$. Diagonalizing $\mathsf A$ and completing
the square in each Gaussian coordinate gives, for
$2|z|\|\mathsf A\|_{\op}<1$,
\[
 \log\E e^{zF}
 =-z\Tr\mathsf A-\tfrac12\log\det(I_N-2z\mathsf A)
   +\tfrac{z^2}{2}b^{\mathsf T}(I_N-2z\mathsf A)^{-1}b.
\]
The two signs of $z$ show that $F$ has an absolute exponential moment.
Expanding the logarithm and the resolvent gives, for $k\ge2$,
\[
 \kappa_k(F)=2^{k-1}(k-1)!\Tr(\mathsf A^k)
          +2^{k-3}k!\,b^{\mathsf T}\mathsf A^{k-2}b.
\]
Let $x_1,\ldots,x_t>0$ be the distinct nonzero eigenvalues of
$\mathsf A^2$, let $n_i$ be their multiplicities, and put
$a_i=\|\operatorname{proj}_{\Ker(\mathsf A^2-x_iI)}b\|^2$.
For $k=2m$, $m\ge2$, the preceding formula becomes
\begin{equation}
 s_m=\sum_{i=1}^t n_ix_i^m
             +\frac m2\sum_{i=1}^t a_ix_i^{m-1}.
 \label{lt:spectral-functional}
\end{equation}

Suppose assertion 1 holds. Theorem~\ref{light:classification} writes
its law as an independent Gaussian plus a centered quadratic series.
For the odd case use
\[
 G\widetilde G=\tfrac12\left[
 \left(\tfrac{G+\widetilde G}{\sqrt2}\right)^2
 -\left(\tfrac{G-\widetilde G}{\sqrt2}\right)^2\right].
\]
The central quadratic cumulant formula \cite{NourdinPolySecond}
yields $s_m=\sum_j\lambda_j^{2m}$.
For an infinite series, this follows from its finite truncations:
$L^2$ convergence at ordinary degree at most two implies convergence
of every fixed moment by hypercontractivity, hence of every cumulant.
Therefore, for $p(z)=\sum_{j=0}^t v_jz^j$,
\[
 v^{\mathsf T}\mathsf H_tv
 =\sum_j\lambda_j^4p(\lambda_j^2)^2\ge0.
\]
The sum is finite or absolutely convergent by square summability.
This proves $1\Rightarrow\mathsf H_t\succeq0$.

For the prescribed target, summing \eqref{lt:spectral-functional}
against $v_jv_k$ and using
$\sum_{j,k}v_jv_k(j+k+2)x^{j+k+1}=(x^2p(x)^2)'$ gives
\begin{equation}
 v^{\mathsf T}\mathsf H_tv
 =\sum_i(n_ix_i^2+a_ix_i)p(x_i)^2
       +\sum_i a_ix_i^2p(x_i)p'(x_i).
 \label{lt:hankel-functional}
\end{equation}
If some $a_i>0$, define
\[
 R(z)=\prod_{j=1}^t(z-x_j),\qquad
 \ell_i(z)=\prod_{j\ne i}\frac{z-x_j}{x_i-x_j},\qquad
 p_c(z)=\ell_i(z)+cR(z).
\]
Then $\deg p_c\le t$, $p_c(x_j)=\delta_{ij}$, and
$p_c'(x_i)=\ell_i'(x_i)+cR'(x_i)$ with $R'(x_i)\ne0$.
Its quadratic form in \eqref{lt:hankel-functional} is
\[
 B_i+c\,a_ix_i^2R'(x_i),\qquad
 B_i=n_ix_i^2+a_ix_i+a_ix_i^2\ell_i'(x_i).
\]
Taking $c=-(B_i+1)/(a_ix_i^2R'(x_i))$ gives the value $-1$.
Thus $\mathsf H_t\succeq0$ forces every $a_i=0$, equivalently
$b\in\Ker\mathsf A$.

If $\mathsf Ab=0$, the linear term depends only on the Gaussian
projection onto $\Ker\mathsf A$ and the quadratic term only on its
orthogonal complement. These terms are independent, so
Theorem~\ref{light:classification} realizes the law at every even
order. This proves $3\Rightarrow2\Rightarrow1$ and all equivalences.
The largest index of $s_m$ in $\mathsf H_t$ is $m=2t+2$; the
largest cumulant order used is therefore $4t+4$.
\end{proof}

\begin{example}[A Hankel obstruction at order twelve]
\label{lt:example-hankel}
For independent standard $X_1,X_2$, consider
\[
 F=(X_1^2-1)+2(X_2^2-1)+X_1.
\]
Here $t=2$ and $s_m=1+4^m+m/2$, so
\[
 \mathsf H_2=
 \begin{pmatrix}
 18&133/2&259\\
 133/2&259&2055/2\\
 259&2055/2&4100
 \end{pmatrix}.
\]
Its leading $2\times2$ determinant is $959/4>0$, but $\det\mathsf
H_2=-324$. The first Hankel test using cumulants of orders $4,6,8$
therefore passes, while the test through order $12$ excludes the law
from every fixed homogeneous chaos closure. More simply, $(X+a)Y$ and
$(X+a)^2-(1+a^2)$, with independent standard $X,Y$ and $a\neq0$, both
fail $\mathsf A b=0$ and are likewise excluded.
\end{example}

\subsection{Joint quadratic Gaussian targets}

\begin{theorem}[Joint realizability of quadratic Gaussian targets]
\label{light:vector}
Let
\[
 F_i=X^\top\mathsf A_iX-\operatorname{tr}\mathsf A_i+b_i^\top X,
 \qquad 1\leq i\leq d,
\]
where $X\sim\mathcal N(0,I_N)$ and the matrices $\mathsf A_i$ are real
symmetric. The vector $F$ is a joint weak limit of vectors whose
coordinates all belong to one fixed homogeneous chaos if and only if
\begin{equation}
\label{lt:vector-criterion}
 \mathsf A_i b_i=0,\qquad
 \mathsf A_i b_j+\mathsf A_j b_i=0\quad(i<j).
\end{equation}
If these conditions hold, such a realization exists at every prescribed
even order $q\ge2$.
\end{theorem}

\begin{proof}
For $u\in\R^d$, put $\mathsf A(u)=\sum_i u_i\mathsf A_i$ and
$b(u)=\sum_i u_i b_i$. A linear combination of coordinates in one
fixed chaos remains in that chaos. Theorem~\ref{light:quadratic}
therefore gives $\mathsf A(u)b(u)=0$ for every $u$ whenever the
joint limit exists. Expanding,
\[
 \mathsf A(u)b(u)
 =\sum_i u_i^2\mathsf A_i b_i
   +\sum_{i<j}u_i u_j(\mathsf A_i b_j+\mathsf A_j b_i).
\]
The coefficients vanish if and only if
\eqref{lt:vector-criterion} holds.

Conversely, suppose these coefficients vanish. For each fixed $u$,
$b(u)\in\Ker\mathsf A(u)$, so the two scalar Gaussian projections
producing $b(u)^{\mathsf T}X$ and
$X^{\mathsf T}\mathsf A(u)X-\Tr\mathsf A(u)$ are independent.
Consequently
\[
 \E e^{\mathrm i\langle u,F\rangle}
 =\E e^{\mathrm i(X^{\mathsf T}\mathsf A(u)X-\Tr\mathsf A(u))}
                          e^{-\|b(u)\|^2/2}.
\]
Let $H\sim N_d(0,C)$, $C_{ij}=b_i^{\mathsf T}b_j$, be independent
of $X$, and put $Q_i=X^{\mathsf T}\mathsf A_iX-\Tr\mathsf A_i$.
Since $u^{\mathsf T}Cu=\|b(u)\|^2$, the last expression is the
characteristic function of $Q+H$. Equality for every $u$ proves
$F\stackrel d=Q+H$ as vectors. For even $q\ge2$, give the coordinates
of $X$ weight $q/2$ and independent Gaussian coordinates representing
$H$ weight $q$. Each component then lies in $\cP_q(E)$, and
Theorem~\ref{rec:realization} gives a simultaneous realization.
\end{proof}

\section{Limits under Ornstein--Uhlenbeck noise}
\label{app:section}

A primitive coordinate of weight $s$ has relaxation rate $s$ under
stationary Ornstein--Uhlenbeck noise. We prove joint convergence of
Malliavin covariances and convergence of the noisy observables on
compact time intervals. Two fourth-chaos sequences then give different
process limits despite identical one-time laws and covariance functions.

\subsection{Malliavin covariances}

On $E=\bigoplus_{s\ge1}E_s$, let $D_s$ and $L_s$ denote the Malliavin
derivative and Ornstein--Uhlenbeck generator in the coordinates of $E_s$.
Set
\begin{equation}
 L_E=\sum_{s\ge1}sL_s,
 \qquad
 \Gamma_E(F,G)=\sum_{s\ge1}s\ip{D_sF}{D_sG}_{E_s}.
 \label{app:weighted-generator}
\end{equation}
Thus $L_EF=-qF$ for $F\in\cP_q(E)$. These operators are well defined on
$\bigoplus_{q=0}^m\cP_q(E)$ for every finite $m$. We use the standard Gaussian
calculus of \cite[Chapter~1]{Nualart} and
\cite[Chapter~2]{NourdinPeccatiBook}.

\begin{theorem}[Joint limits of polynomial and derivative expressions]
\label{app:derivatives}
Let the components of $F_n=(F_{n,a})_{a=1}^d$ have chaos orders at most
$Q$ and bounded second moments. Fix finitely many expressions formed
from these components by addition, multiplication and application of
$L$. Enlarge the degree bound to include these expressions and select
their kernels as in Theorem~\ref{str:law-transfer}. Along the resulting
subsequence, all these expressions converge jointly, with all joint
moments, to the corresponding expressions formed from the representation
$F=(F_a)_{a=1}^d$ and $L_E$. In particular, including the indicated
derivative expressions in this selection gives
\begin{equation}
 \left(F_n,\big(\ip{DF_{n,a}}{DF_{n,b}}\big)_{a,b}\right)
 \weak
 \left(F,\big(\Gamma_E(F_a,F_b)\big)_{a,b}\right).
 \label{app:derivative-limit}
\end{equation}
\end{theorem}

\begin{proof}
Include all subexpressions and their homogeneous components in the
finite list. Their degrees and second moments are bounded by induction:
for $G_n,H_n$ of degrees at most $r,s$,
\[
 \|G_nH_n\|_2\le\|G_n\|_4\|H_n\|_4
       \le3^{(r+s)/2}\|G_n\|_2\|H_n\|_2,
 \qquad \|LG_n\|_2\le r\|G_n\|_2.
\]
Addition preserves boundedness as well.

\emph{Represented limits.}
Every expression $G_n$ in the finite list is a finite sum of chaos
components,
\begin{equation}
 G_n=\sum_{m=0}^{Q'}I_m\!\left(\frac{g_{n,m}}{\sqrt{m!}}\right),
 \label{app:expression-chaos}
\end{equation}
for some $Q'$. This follows from the product formula
\eqref{pre:product-formula} and from
$LI_m(g/\sqrt{m!})=-mI_m(g/\sqrt{m!})$. Orthogonality gives $\|g_{n,m}\|^2\le\|G_n\|_2^2$, so these
kernel sequences are bounded.
We include all these kernel sequences in the family, and all constant
components among the converging scalars. By
Lemma~\ref{str:enlargement}, this does not change the representation
of the original components. We define
\begin{equation}
 R(G)=\sum_{m=0}^{Q'}\mathbb U_m^{-1}[g_m],
 \qquad \mathbb U_m^{-1}[g_m]\in\cP_m(E).
 \label{app:represented-limit}
\end{equation}
By Theorem~\ref{str:law-transfer}, all the $G_n$ in the list converge
jointly to $(R(G))_G$, with all joint moments.

\emph{Sums and products.}
Clearly $R(G+G')=R(G)+R(G')$. Suppose that $G,G'$ and $GG'$ belong to
the list. Along the subsequence,
\[
 (G_n,G'_n,G_nG'_n)\Longrightarrow\bigl(R(G),R(G'),R(GG')\bigr).
\]
The laws on the left are carried by the closed set
$\{(a,b,c):c=ab\}$. By the portmanteau theorem, so is the limit law.
Hence
\begin{equation}
 R(GG')=R(G)\,R(G')\qquad\text{almost surely}.
 \label{app:product-rule}
\end{equation}

\emph{The generator.}
A Wick monomial $:\!W(h_1)\cdots W(h_k)\!:$ with $h_i\in E_{s_i}$ is
an eigenfunction of $L_s$ with eigenvalue minus the number of its
factors in $E_s$. It is therefore an eigenfunction of $L_E$ with
eigenvalue $-\sum_is_i$, its weighted degree. Thus $L_E=-m$ on
$\cP_m(E)$. On the original side, $L$ multiplies the chaos component
of order $m$ by $-m$. Comparing with \eqref{app:represented-limit},
\begin{equation}
 R(LG)=-\sum_mm\,\mathbb U_m^{-1}[g_m]=L_E\,R(G).
 \label{app:generator-rule}
\end{equation}

\emph{Malliavin covariances.}
By induction on the finite expression, \eqref{app:product-rule} and
\eqref{app:generator-rule} show that every expression converges to
the same expression formed from $F=R(F_\cdot)$ and $L_E$. For
finite-chaos variables,
\begin{equation}
 \langle DX,DY\rangle=\tfrac12\bigl\{L(XY)-XLY-YLX\bigr\}.
 \label{app:carre-du-champ}
\end{equation}
The same identity holds for each $L_s$, with $\langle D_sX,D_sY\rangle$.
Multiplying by $s$ and summing over $s$ gives it for $L_E$ and
$\Gamma_E$. Apply \eqref{app:product-rule} and \eqref{app:generator-rule} to
\eqref{app:carre-du-champ}, with $X=F_{n,a}$ and $Y=F_{n,b}$. This gives
\[
 R\bigl(\langle DF_{n,a},DF_{n,b}\rangle\bigr)
 =\tfrac12\bigl\{L_E(F_aF_b)-F_aL_EF_b-F_bL_EF_a\bigr\}
 =\Gamma_E(F_a,F_b).
\]
Together with the joint convergence above, this proves
\eqref{app:derivative-limit}.
\end{proof}

\subsection{Stationary process limits}

For each $n$, let $W_n(t)$ be a stationary isonormal
Ornstein--Uhlenbeck process over $H_n$, with
$\E[W_n(t,h)W_n(u,k)]=e^{-|t-u|}\ip hk$.
For rescaled kernels $x_{n,a,q}$, define
\begin{equation}
 F_{n,a}(t)=\sum_{q=0}^Q
 I_q^{W_n(t)}\!\left(\frac{x_{n,a,q}}{\sqrt{q!}}\right).
 \label{app:OU-observables}
\end{equation}

\begin{theorem}[Ornstein--Uhlenbeck process limits]
\label{app:process}
Assume $\sup_n\sum_{a,q}\norm{x_{n,a,q}}^2\le R^2$ and take the
subsequence of Theorem~\ref{str:law-transfer}, including convergence
of the constant components. For every $T<\infty$, the processes in
\eqref{app:OU-observables} converge in law in $C([0,T],\R^d)$.
The limit is obtained by evaluating the representing Hermite expansions
at independent stationary standard Gaussian processes $G_{s,j}$ with
\begin{equation}
 \E[G_{s,j}(t)G_{r,k}(u)]
 =\1_{\{s=r,\,j=k\}}e^{-s|t-u|}.
 \label{app:limit-OU}
\end{equation}
The generator of this underlying Gaussian process is $L_E$.
\end{theorem}

\begin{proof}
Take $K=L^2(\R)$ and
$e_t(v)=\sqrt2e^{-(t-v)}\1_{\{v\le t\}}$.
Then $\|e_t\|=1$ and $\langle e_t,e_u\rangle=e^{-|t-u|}$.
The maps $\iota_t:H_n\to H_n\otimes K$, $h\mapsto h\otimes e_t$,
realize $W_n(t,h)$ as a single isonormal process evaluated at
$\iota_th$.
For primitive representatives $x_n,y_n$ of order $s$,
\[
 \langle\iota_t^{\otimes s}x_n,\iota_u^{\otimes s}y_n\rangle
 =e^{-s|t-u|}\langle x_n,y_n\rangle,
\]
and, for $1\le r<s$,
\[
 \|(\iota_t^{\otimes s}x_n)\otimes_r
                  (\iota_t^{\otimes s}x_n)\|
 =\|x_n\otimes_rx_n\|\longrightarrow0.
\]
For finitely many times and primitive factors, these covariance
limits and the Gaussian convergence criterion in
Lemma~\ref{str:joint-normal} give joint Gaussian limits with
covariance \eqref{app:limit-OU}. Different weights have zero
covariance because their original Wiener chaos orders differ.
Since $\iota_t^{\otimes q}$ commutes with symmetrization and
$\diamond$, the Wick-polynomial argument in the proof of
Theorem~\ref{str:representation} applies simultaneously at these
times. This proves finite-dimensional convergence for algebraic
weighted expansions.

For general expansions, fix times $t_1,\ldots,t_m$ and choose
algebraic $P_{a,q}^{(j)}$ such that
\[
 \delta_j^2=\sum_{a,q}
       \|\xi_{a,q}-\mathbb U_qP_{a,q}^{(j)}\|^2\longrightarrow0.
\]
Let $F_n^{(j)}(t)$ and $F^{(j)}(t)$ be the corresponding original
and limiting polynomial observables. Norm preservation by $\iota_t$
and stationarity give
\begin{align*}
 \lim_n\E\sum_{\nu=1}^m
        |F_n(t_\nu)-F_n^{(j)}(t_\nu)|^2&=m\delta_j^2,\\
 \E\sum_{\nu=1}^m
        |F(t_\nu)-F^{(j)}(t_\nu)|^2&=m\delta_j^2.
\end{align*}
The limit $F(t)$ here is defined by the $L^2$ Hermite expansions at
each fixed time. A bounded $1$-Lipschitz test of the $m$ evaluations
has limiting expectation error at most $2\sqrt m\,\delta_j$, since
the polynomial approximants converge. Letting $j\to\infty$ proves
all finite-dimensional limits.

For path convergence, the Hermite covariance identity gives
\[
 \E|F_{n,a}(t)-F_{n,a}(u)|^2
 =2\sum_{q=1}^Q\|x_{n,a,q}\|^2(1-e^{-q|t-u|})
 \le2QR^2|t-u|.
\]
The increment belongs to the first $Q$ chaoses on $H_n\otimes K$,
so hypercontractivity yields
\begin{equation}
 \E|F_{n,a}(t)-F_{n,a}(u)|^4
 \le4\,3^{2Q}Q^2R^4|t-u|^2.
 \label{app:OU-fourth-increment}
\end{equation}
Put $\mathcal Z(t)=(G_{s,j}(t))_{s,j}$ and
$\wt(\alpha)=\sum_{s,j}s\alpha_{s,j}$. Its Hermite covariance is
\[
 \E[H_\alpha(\mathcal Z(t))H_\beta(\mathcal Z(u))]
 =\1_{\{\alpha=\beta\}}\alpha!e^{-\wt(\alpha)|t-u|}.
\]
The same two increment bounds therefore hold for $F$, by $L^2$
approximation and hypercontractivity.
Kolmogorov's criterion gives continuous versions. The uniform fourth
moment bound and tightness at time zero give tightness of $(F_n)$ in
$C([0,T],\R^d)$ \cite[Section~7]{Billingsley}.
Every subsequential path limit has the finite-dimensional laws just
identified; evaluations at rational times determine the Borel law on
this space. Hence the full selected subsequence converges to $F$.
The independent coordinate semigroup acts on $H_\alpha$ by
$e^{-t\wt(\alpha)}$, so its generator is $L_E$.
\end{proof}

\subsection{Additive functionals and two distinct process limits}

\begin{corollary}[Asymptotic variance of an additive functional]
\label{app:additive-functional}
Let $\mathcal Z(t)=(G_{s,j}(t))_{s,j}$ be the stationary Gaussian
coordinate family in \eqref{app:limit-OU}. For a centered observable
$h=\sum_{r=1}^m h_r$, where $h_r\in\cP_r(E)$, one has
\[
 \frac1{\sqrt T}\int_0^T h(\mathcal Z(t))\,dt
 \weak N(0,\sigma_h^2)\qquad(T\to\infty),
\]
with
\begin{equation}
 \sigma_h^2
 =2\ip{h}{(-L_E)^{-1}h}_{L^2}
 =2\sum_{r=1}^m\frac{\norm{h_r}_2^2}{r}.
 \label{app:asymptotic-variance}
\end{equation}
\end{corollary}
\begin{proof}
Let $P_t$ be the semigroup of $\mathcal Z$ on its product Gaussian
space. The Hermite basis diagonalizes $P_t$:
$P_tH_\alpha=e^{-t\wt(\alpha)}H_\alpha$.
This proves reversibility, and for any centered $g\in L^2$,
\[
 \|P_tg\|_2^2
 =\sum_{\alpha\ne0}e^{-2t\wt(\alpha)}
                       |c_\alpha|^2\alpha!
 \le e^{-2t}\|g\|_2^2.
\]
Thus the stationary process is ergodic. For the given $h$,
\[
 g=(-L_E)^{-1}h=\sum_{r=1}^m r^{-1}h_r\in\operatorname{Dom}(L_E),
 \qquad
 \langle h,(-L_E)^{-1}h\rangle=\sum_{r=1}^m r^{-1}\|h_r\|_2^2<\infty.
\]
This verifies reversibility, ergodicity and the $H_{-1}$ condition
in the additive-functional central limit theorem of
Kipnis and Varadhan \cite{KipnisVaradhan}. It gives the asserted
normal limit. The variance can also be read directly from
\[
 \E[h(\mathcal Z(0))h(\mathcal Z(t))]
 =\langle h,P_th\rangle=\sum_{r=1}^m e^{-rt}\|h_r\|_2^2.
\]
Indeed, stationarity and Fubini give the exact identity
\[
 \Var\left(\frac1{\sqrt T}\int_0^T h(\mathcal Z(t))\,dt\right)
 =2\sum_{r=1}^m\|h_r\|_2^2
       \left(\frac1r-\frac{1-e^{-rT}}{Tr^2}\right).
\]
Its limit is \eqref{app:asymptotic-variance}.
\end{proof}

\begin{proposition}[Two distinct fourth-chaos process limits]
\label{app:fourth-example}
\label{app:example-section}\label{app:time-averages}
Let $U,(X_i),(Y_i),(Z_i)$ be independent standard Gaussian variables,
and put
\begin{align*}
 V_n&=\frac1{\sqrt{6n}}\sum_{i=1}^nH_3(X_i),
 &B_n&=\frac1{\sqrt{2n}}\sum_{i=1}^nH_2(Y_i),
 &C_n&=\frac1{\sqrt{2n}}\sum_{i=1}^nH_2(Z_i),\\
 F_n&=UV_n,
 &\widetilde F_n&=B_nC_n.
\end{align*}
Both sequences belong to the fourth homogeneous chaos and have variance
one. Under independent stationary rate-one Ornstein--Uhlenbeck dynamics
on the original coordinates, their process limits in $C([0,T],\R)$,
for every $T<\infty$, are respectively
\begin{equation}
 \mathsf X(t)=U_1(t)V_3(t),
 \qquad
 \mathsf Y(t)=B_2(t)C_2(t),
 \label{app:example-processes}
\end{equation}
where the factors are independent standard stationary Gaussian
processes and the subscripts specify their relaxation rates.
The two limits have the same product-normal marginal law and covariance
$e^{-4|t-u|}$, but different process laws. Their centered-square time
averages satisfy
\begin{equation}
 \begin{split}
 \frac1{\sqrt T}\int_0^T(\mathsf X(t)^2-1)\,dt
   &\weak N(0,11/3),\\
 \frac1{\sqrt T}\int_0^T(\mathsf Y(t)^2-1)\,dt
   &\weak N(0,3).
 \end{split}
 \label{app:two-average-limits}
\end{equation}
\end{proposition}

\begin{proof}
The factors defining each product use disjoint Gaussian coordinates.
The product formula therefore places $F_n$ and $\widetilde F_n$ in
the fourth chaos, and independence gives their unit variances.
The central limit theorem gives independent standard Gaussian limits
$U,V,B,C$, with primitive weights $(1,3)$ for $UV$ and $(2,2)$ for
$BC$. Theorem~\ref{app:process} yields
\eqref{app:example-processes}. Independence gives the common marginal
law and covariance.

The Malliavin derivatives already distinguish the two limits.
The law of large numbers gives
\[
 \norm{DV_n}^2=\frac3{2n}\sum_iH_2(X_i)^2\longrightarrow3,
 \qquad
 \norm{DB_n}^2=\frac2n\sum_iY_i^2\longrightarrow2,
\]
and similarly $\norm{DC_n}^2\to2$. Orthogonality of the coordinate
spaces then gives
\begin{align}
 (F_n,\norm{DF_n}^2)&\weak(UV,V^2+3U^2),\notag\\
 (\widetilde F_n,\norm{D\widetilde F_n}^2)
 &\weak(BC,2B^2+2C^2).
 \label{app:example-joint-limits}
\end{align}
The second coordinates have mean $4$ but variances $20$ and $16$.

For the process laws, the Gaussian moment formula gives, for $t>0$,
\begin{align*}
 \E[\mathsf X(0)^2\mathsf X(t)^2]
   &=(1+2e^{-2t})(1+2e^{-6t}),\\
 \E[\mathsf Y(0)^2\mathsf Y(t)^2]
   &=(1+2e^{-4t})^2.
\end{align*}
Their difference is $2e^{-2t}(1-e^{-2t})^2>0$.
Finally, for independent Gaussian coordinates of rates $(a,b)$, write
\[
 h(u,v)=u^2v^2-1
       =H_2(u)+H_2(v)+H_2(u)H_2(v).
\]
The three orthogonal terms have squared norms $2,2,4$ and generator
eigenvalues $-2a,-2b,-2(a+b)$. Corollary~\ref{app:additive-functional}
therefore gives
\[
 \sigma_h^2=\frac2a+\frac2b+\frac4{a+b}.
\]
The rates $(1,3)$ and $(2,2)$ give $11/3$ and $3$, respectively.
\end{proof}

\begin{remark}[Order of limits]
The limits in \eqref{app:two-average-limits} concern the limiting
processes. By the continuous mapping theorem at each fixed $T$, they
also describe the successive limits $n\to\infty$, then $T\to\infty$,
for the original noisy observables. No simultaneous limit with
$T=T_n\to\infty$ is asserted.
\end{remark}

\section{Nonnegative kernels and Gaussian sums over triangles}
\label{pos:section}

For coordinatewise nonnegative kernels, every scalar contraction
defined by a regular graph is nonnegative. The contractions associated
with simple graphs depend only on the weight-one Gaussian coordinates
in the limiting representation.

\begin{definition}[Nonnegative coordinates and regular graph contractions]
\label{pos:definition}
A symmetric kernel $x\in H^{\odot q}$ has \emph{nonnegative coordinates}
if there is an orthonormal basis $(e_i)$ of $H$ in which all its tensor
coordinates are nonnegative. For a finite simple $q$-regular graph
$\Gamma=(V,\mathcal E)$ and a symmetric tensor $t\in H^{\odot q}$, set
\[
 \Gamma(t)=\sum_{(i_e)_{e\in\mathcal E}}
       \prod_{v\in V}t\big((i_e)_{e\ni v}\big).
\]
This is the complete contraction with one copy of $t$ at every vertex
and one contracted index on every edge. It is independent of the basis.
Lemma~\ref{str:diagram-bound} gives
$|\Gamma(t)|\le\|t\|^{|V|}$ and, applied to absolute coordinates,
defines the sum in infinite dimension.
\end{definition}

\begin{theorem}[Positivity of the component over $E_1$]
\label{pos:graph-obstruction}
Let $q\ge2$, and let $(x_n)$ be a bounded sequence of rescaled kernels
with nonnegative coordinates in $H_n^{\odot q}$. In a representation
from Theorem~\ref{str:representation}, write
\[
 \mathbb U_q^{-1}[x]=I_q^W(t/\sqrt{q!})+F_{\mathrm{higher}},
 \qquad t\in E_1^{\odot q},
\]
where $F_{\mathrm{higher}}$ contains the Hermite monomials involving a
coordinate of weight at least two. Thus $t$ is the rescaled kernel
of the component over $E_1$.
Then, for every finite simple $q$-regular graph $\Gamma$, along the
subsequence of Theorem~\ref{str:representation},
\begin{equation}
 \lim_n\Gamma(x_n)=\Gamma(t)\ge0.
 \label{pos:graphlimit}
\end{equation}
\end{theorem}

\begin{proof}
\emph{Positivity before the limit.}
Let $(e_i)$ be a basis in which $x_n$ has nonnegative coordinates.
Then $\Gamma(x_n)$ is a sum of products of such coordinates. The sum
converges absolutely, since replacing each coordinate by its absolute
value does not change the norm and Lemma~\ref{str:diagram-bound} applies.
The value is independent of the basis. Hence
\begin{equation}
 \Gamma(x_n)\ge0 .
 \label{pos:finite-positive}
\end{equation}

\emph{Algebraic inputs.}
Let $p_n$ be a finite combination of diamond products
$h_{1,n}\diamond\cdots\diamond h_{k,n}$ with $[h_{i,n}]\in\cG_{s_i}$
and $\sum_is_i=q$. By Theorem~\ref{str:contractions}, applied to the
closed diagram $\Gamma$,
\begin{equation}
 \Gamma(p_n)-B_{\Gamma,n}\longrightarrow0,
 \label{pos:algebraic-limit}
\end{equation}
where $B_{\Gamma,n}$ keeps only the terms in which every factor is
paired in all its slots with a single factor of the same weight.
Expand $\Gamma(p_n)$ multilinearly over the monomials placed at the
vertices. Consider a monomial at a vertex $v$ containing a factor $h$
of weight $s\ge2$. The $q$ slots at $v$ correspond to the $q$ edges
at $v$. Since $\Gamma$ is simple, these edges lead to $q$ distinct
neighbours. The $s$ slots of $h$ therefore lead to $s\ge2$ distinct
vertices. The factor $h$ cannot be paired in all its slots with a
single factor. Every such term is absent from $B_{\Gamma,n}$.

The surviving terms carry, at every vertex, a monomial whose $q$
factors all have weight one. For weight one, $h_{i,n}\in H_n$ is a
vector. The contraction of two slots is the scalar product
$\langle h_{i,n},h_{j,n}\rangle$. It converges to
$\langle h_i,h_j\rangle_{E_1}$. Moreover
$I_q^W(h_1\diamond\cdots\diamond h_q/\sqrt{q!})
=\,:\!W(h_1)\cdots W(h_q)\!:$ for $h_i\in E_1$. The weight-one part of
$\mathbb U_q^{-1}[p]$ is therefore $I_q^W(t_p/\sqrt{q!})$, where
$t_p\in E_1^{\odot q}$ is the same combination of the diamond products
$h_1\diamond\cdots\diamond h_q$. Consequently
\begin{equation}
 \lim_n\Gamma(p_n)=\Gamma(t_p).
 \label{pos:algebraic-value}
\end{equation}

\emph{Approximation.}
Choose algebraic $p^{(j)}$ with $\|[x]-[p^{(j)}]\|\le1/j$. The map
$[x]\mapsto t$ is the orthogonal projection onto the weight-one part,
so $\|t-t_{p^{(j)}}\|\le1/j$. The associated multilinear contraction
satisfies $|\Gamma(z_1,\ldots,z_k)|\le\prod_v\|z_v\|$ by
Lemma~\ref{str:diagram-bound}. The telescoping identity
\eqref{rec:telescoping} therefore gives
\[
 \limsup_n|\Gamma(x_n)-\Gamma(p^{(j)}_n)|\le\frac{k\,C^{k-1}}j,
 \qquad
 |\Gamma(t)-\Gamma(t_{p^{(j)}})|\le\frac{k\,C^{k-1}}j,
\]
where $k=|V|$ and $C=1+\|[x]\|$ bounds the sequence-space norms and
their projected norms. With
\eqref{pos:algebraic-value}, letting $j\to\infty$ shows that
$\lim_n\Gamma(x_n)$ exists and equals $\Gamma(t)$. By
\eqref{pos:finite-positive}, this limit is nonnegative.
\end{proof}

\begin{example}[A cubic representation excluded by positivity]
\label{pos:cubic}
Let $e_1,e_2,e_3$ be orthonormal, and let $T_{a,b}$ be the symmetric
coefficient tensor of the cubic form $z_1^3+az_1z_2^2+bz_1z_3^2$.
Its nonzero coordinates are
\[
 \begin{gathered}
 (T_{a,b})_{111}=1,\\
 (T_{a,b})_{122}=(T_{a,b})_{212}=(T_{a,b})_{221}=a/3,\\
 (T_{a,b})_{133}=(T_{a,b})_{313}=(T_{a,b})_{331}=b/3.
 \end{gathered}
\]
For the complete graph on four vertices,
\begin{equation}
 K_4(T_{a,b})=1+\frac{a^4+4a^3+b^4+4b^3}{27}.
 \label{pos:K4}
\end{equation}
Indeed, in a nonzero edge labeling the edges labeled $2$ or $3$ form
vertex-disjoint monochromatic cycles. Besides the all-$1$ labeling,
$K_4$ allows four triangles and three four-cycles of each color.
Their contributions are $4(a/3)^3+3(a/3)^4$ and the analogous expression
in $b$, proving \eqref{pos:K4}.

In particular, $K_4(T_{-3,-3})=-1$. The corresponding Wick polynomial
has rescaled kernel $x_*=\sqrt6\,T_{-3,-3}$ and is
\begin{equation}
 F_*:=I_3^W(x_*/\sqrt6)
   =H_3(G_1)-3G_1H_2(G_2)-3G_1H_2(G_3),
 \qquad \mathbb E F_*^2=42.
 \label{pos:Fstar}
\end{equation}
Since $K_4(x_*)=36K_4(T_{-3,-3})=-36$, no sequence of nonnegative
kernels can have a component over $E_1$
orthogonally equivalent to $x_*$, or to its negative, in its
contraction-compatible representation. Adding components of types
$1+2$ and $3$ does not remove this obstruction. The prescribed
representation is realizable by signed smooth kernels, by
Theorem~\ref{rec:realization}. This conclusion concerns the
representation: it does not exclude another representation with the
same scalar law and a different tensor over $E_1$.
\end{example}

We record the following standard consequence of the diagram formula.

\begin{corollary}[Cumulant positivity for nonnegative kernels]
\label{pos:cumulants}
Let $F_{n,a}=I_{q_a}(x_{n,a}/\sqrt{q_a!})$, $1\le a\le d$, have bounded
total variance, and suppose that all the kernels at each $n$ have
nonnegative coordinates in a common orthonormal basis. Every joint
cumulant of order at least two is nonnegative. If $F_n$ converges in
law, the same conclusion holds for its limit.
\end{corollary}

\begin{proof}
Write $f_{n,a}=x_{n,a}/\sqrt{q_a!}$. For indices
$a_1,\ldots,a_k$, consider the $k$ blocks of legs, the $i$-th block
having $q_{a_i}$ legs. The diagram formula
\cite[Chapter~7]{PeccatiTaqqu} states
\begin{equation}
 \kappa\bigl(F_{n,a_1},\ldots,F_{n,a_k}\bigr)
 =\sum_{\sigma}\mathcal C_\sigma\bigl(f_{n,a_1},\ldots,f_{n,a_k}\bigr).
 \label{pos:cumulant-diagram}
\end{equation}
The sum runs over the complete pairings $\sigma$ of all legs that
satisfy two conditions. No two legs of the same block are paired, and
the graph on the blocks induced by $\sigma$ is connected. The term
$\mathcal C_\sigma$ is the corresponding complete contraction. The
moment formula is the same sum without the connectedness condition.
Formula \eqref{pos:cumulant-diagram} follows from it by Möbius
inversion over partitions of the blocks.

In the common basis, $\mathcal C_\sigma$ is a sum of products of
nonnegative coordinates. It converges absolutely by
Lemma~\ref{str:diagram-bound}, applied to the coordinatewise absolute
values. Every term of \eqref{pos:cumulant-diagram} is therefore
nonnegative, and so is the cumulant.

Suppose that $F_n\Rightarrow F$. By Lemma~\ref{pre:tightness}, all
joint moments of fixed order converge. A joint cumulant of order $k$
is a polynomial in the joint moments of order at most $k$. It
therefore converges, and
\[
 \kappa(F_{a_1},\ldots,F_{a_k})
 =\lim_n\kappa(F_{n,a_1},\ldots,F_{n,a_k})\ge0.
\]
\end{proof}

\begin{definition}[Gaussian sum indexed by triangles]
\label{pos:triangles}
For a finite simple graph $\mathsf G$ with triangle set
$\mathcal T(\mathsf G)$ of size $m>0$, let $(X_v)_{v\in V(\mathsf G)}$
be independent standard Gaussian variables and define
\[
 T(\mathsf G)=\frac1{\sqrt m}
     \sum_{\{i,j,k\}\in\mathcal T(\mathsf G)}X_iX_jX_k.
\]
Distinct vertex triples give orthogonal monomials, hence
$\mathbb E T(\mathsf G)^2=1$. Its rescaled third-order kernel is
\[
 x_{\mathsf G}=\frac{\sqrt6}{\sqrt m}
       \sum_{\{i,j,k\}\in\mathcal T(\mathsf G)}
          \operatorname{Sym}(e_i\otimes e_j\otimes e_k),
 \qquad
 T(\mathsf G)=I_3(x_{\mathsf G}/\sqrt6).
\]
In particular, $x_{\mathsf G}$ has nonnegative coordinates and norm one.
\end{definition}

\begin{example}[The three weight partitions in order three]
\label{pos:triangletypes}
The following elementary graph sequences exhibit the types $3$,
$1+2$ and $1+1+1$. They use only the ordinary central limit theorem
\cite{Billingsley} and the Wiener isometry.
\begin{enumerate}
\item For $m$ vertex-disjoint triangles, $T(\mathsf G)$ is the
normalized sum of $m$ independent copies of $X_1X_2X_3$, so it
converges to a standard Gaussian of weight $3$.
\item For a windmill consisting of $m$ triangles with one common hub
$h$ and otherwise disjoint vertices,
\[
 T(\mathsf G)=X_h\frac1{\sqrt m}\sum_{j=1}^m X_{a_j}X_{b_j}
       \ \Longrightarrow\ GN,
\]
where $G,N$ are independent standard Gaussians of weights
$1,2$. Independence holds before taking the limit because the hub is
disjoint from the blades.
\item For the complete graph $K_n$, put
$u_n=n^{-1/2}\sum_{i=1}^ne_i$. The rescaled kernel is the unit tensor
constant on ordered triples of distinct vertices. Consequently,
\[
 \|x_{K_n}-u_n^{\otimes3}\|^2
       =2-2\sqrt{\frac{n(n-1)(n-2)}{n^3}}\longrightarrow0.
\]
Thus $T(K_n)-H_3(W_n(u_n))/\sqrt6\to0$ in $L^2$, and
$T(K_n)\Rightarrow H_3(G)/\sqrt6$.
\end{enumerate}
\end{example}

\section{Further questions}
\label{questions:section}

\paragraph{Optimality of the extraction rate.}
The smooth Gaussian replacement bound in Corollary~\ref{ext:rates}
has order $J^{-1/2}$. Is this rate optimal for the extraction procedure
of Theorem~\ref{ext:finite}, uniformly over inputs of bounded total variance, fixed maximal order
and fixed vector dimension? A broader minimax question must specify the
permitted lower-order factor spaces and require that
the residual approach the primitive projection. A bound on the
dimension of an arbitrary active subspace alone does not specify a
meaningful complexity constraint: that subspace could simply contain
all input kernels.

\paragraph{A triangle obstruction at the level of laws.}
Does $F_*/\sqrt{42}$ from \eqref{pos:Fstar} occur as a weak limit of
$T(\mathsf G_n)$ along deterministic graphs with
$|\mathcal T(\mathsf G_n)|\to\infty$? Equivalently, can $F_*$ be obtained
after the deterministic rescaling $\sqrt{42}\,T(\mathsf G_n)$?
Theorem~\ref{pos:graph-obstruction} excludes its displayed cubic tensor, but
does not exclude a different weighted representation of the same law.
The displayed $K_4$ contraction alone does not distinguish its law
from all possible alternative representations.

\paragraph{Which cubic tensors over $E_1$ are realizable by graphs?}
Consider the $1+1+1$ tensors in contraction-compatible representations
of limits of Gaussian sums indexed by triangles. Is their class, up to
orthogonal identification, exactly the Hilbert tensor closure of the
nonnegative symmetric kernels, subject to the variance bound
$\|t\|^2\le1$? Here $t$ is rescaled as in
Theorem~\ref{pos:graph-obstruction}, and closure means norm convergence after isometric
embeddings into a common Hilbert space. In the spatial model
$L^2([0,1])$, nonnegative symmetric kernels can
be approximated in $L^2([0,1]^3)$ by nonnegative finite step kernels.
This approximation in a fixed spatial model does
not itself supply a positive coordinate representation after passage
to the limiting kernel space.
One must also determine the relation of this class to the potentially
larger set
\[
 \{t:\Gamma(t)\ge0\text{ for every finite simple cubic graph }\Gamma\}.
\]
The positivity inequalities are necessary; neither their sufficiency
nor preservation of a positive coordinate realization under extraction
is proved by the argument above.

\paragraph{Does the law identify the primitive variance?}
For $\xi\in\cH_q$, is
$\|P_{\cG_q}\xi\|^2$ determined by the law of
$F=\mathbb U_q^{-1}\xi$? It is invariant under the degree-preserving
isometries of the representation, but this does not settle invariance
under equality of scalar laws. There is a simple lower bound involving
the weighted carr\'e du champ,
\[
 \Gamma_E(F,F)=\sum_{s\ge1}s\|D_sF\|^2
       \ge q\|P_{\cG_q}\xi\|^2,
\]
because the linear coordinates of weight $q$ do not occur in
the remaining terms. Equality at the essential infimum need not hold,
as the following example shows.

\begin{example}[The essential infimum does not recover primitive variance]
\label{questions:gamma-counterexample}
Let $X,Y,U,V,Z$ be independent standard Gaussians, with $X,Y$ of
weight one and $U,V,Z$ of weight two. The polynomial
\[
 F=(X^2-1)U+(Y^2-1)V+XYZ
\]
has weighted degree four and no primitive linear component. Thus its
primitive variance is zero. On the other hand,
\begin{align*}
 \Gamma_E(F,F)
   &=(2XU+YZ)^2+(2YV+XZ)^2\\
   &\quad+2\big[(X^2-1)^2+(Y^2-1)^2+X^2Y^2\big].
\end{align*}
Writing $u=X^2$ and $v=Y^2$ gives the identity
\[
 (u-1)^2+(v-1)^2+uv
   =\frac34\left(u+v-\frac43\right)^2
       +\frac14(u-v)^2+\frac23.
\]
The global minimum $4/3$ of $\Gamma_E(F,F)$ is attained when
$X^2=Y^2=2/3$ and $U=V=Z=0$. Since the joint Gaussian law has full
support and this function is continuous,
\[
 \operatorname*{ess\,inf}\Gamma_E(F,F)=\frac43,
 \qquad \|P_{\cG_4}\xi\|^2=0.
\]
Consequently, the proposed identity
$\|P_{\cG_q}\xi\|^2=q^{-1}\operatorname*{ess\,inf}\Gamma_E(F,F)$
fails already at $q=4$.
\end{example}

\section*{Use of generative artificial intelligence}
The author used ChatGPT Work/Codex,
principally GPT-5.6 Sol Ultra, for critical reading, stress-testing of
arguments, detection of errors and counterexamples, and bibliographic orientation.  All mathematical claims and the final text were checked and approved by the
author.   The originality of the work and full responsibility for its
contents remain with the author.

\end{document}